\documentclass[12pt]{amsart}
\renewcommand{\bar}{\overline}
\usepackage{amsmath,amsthm,amscd,euscript,amssymb}
\usepackage{graphicx}
\usepackage{epstopdf}
\usepackage{color}
\usepackage[most]{tcolorbox}

\newcommand{\myexample}[2]{
   \begin{tcolorbox}[breakable,colback=black!5!white,colframe=black,title={
   #1}]
        #2
    \end{tcolorbox}
}

\renewcommand{\bar}{\overline}
\def \r{\mathbb R}

\def \q{\mathbb Q}
\def \z{\mathbb Z}

\def \M{\mathcal M}

\def \pyr{\hbox{\rm Pyr}}
\def \GL{\hbox{\rm GL}}
\def \SL{\hbox{\rm SL}}
\def \id{\hbox{\rm Id}}
\def \dist{\hbox{\rm d}}
\def \vol{\hbox{\rm Vol}}
\def \area{\hbox{\rm Area}}

\def \starr{\hbox{\rm star}}

\DeclareMathOperator{\Ahull}{A-hull}

\DeclareMathOperator{\sgn}{sgn} 

\newtheorem{theorem}{Theorem}[section]
\newtheorem{lemma}[theorem]{Lemma}

\newtheorem{proposition}[theorem]{Proposition}
\newtheorem{corollary}[theorem]{Corollary}
\theoremstyle{remark}
\newtheorem{remark}[theorem]{Remark}
\theoremstyle{definition}
\newtheorem{definition}[theorem]{Definition}
\newtheorem{example}[theorem]{Example}
\newtheorem{problem}{Problem}
\newtheorem{conjecture}[problem]{Conjecture}

\title{On a periodic Jacobi-Perron type algorithm}

\author{Oleg Karpenkov}

\date{29 January 2021}


\email[Oleg Karpenkov]{karpenk@liverpool.ac.uk}

\begin{document}
\input{epsf}

\begin{abstract}
In this paper we introduce a new modification of
the Jacobi-Perron algorithm in the three dimensional case
and prove its periodicity for the case of totally-real conjugate cubic vectors.
This provides an answer in the totally-real case to the question of algebraic periodicity for cubic irrationalities posed in 1848 by Ch.~Hermite.
\end{abstract}

\maketitle
\tableofcontents

\section*{Introduction}

Recall that a cubic number is a root of a cubic polynomial with integer coefficients
irreducible over the field $\q$.
A cubic number is called {\it totally-real} if the cubic integer polynomial defining it
has three real roots.
We say that a vector in $\r^3$ is {\it cubic} if its coordinates span a cubic field.
In this paper we introduce a new Jacobi-Perron type algorithm (which we call the $\sin^2$-algorithm)
and prove its periodicity in the totally-real case of cubic vectors.
This provides a solution to Hermite's problem on cubic periodicity in the totally-real case.

\vspace{2mm}

The study of periodic representations has a long history.
It starts with the invention of the Euclidean Algorithm  (in around 300 B.C.)
by ancient Greeks that was designed
for finding the greatest common divisor of two integer numbers.
After the discovery of continued fractions by J.~Wallis in 1695 the Euclidean algorithm
was adapted to pairs of arbitrary real numbers.
Finally, in 1770 in~\cite{Lagrange1770}  J.-L.~Lagrange showed that
a real number is a root of an irreducible quadratic polynomial
if and only if its continued fraction is periodic.

\vspace{2mm}

The question ofsimilar periodic representations of cubic irrationalities was
posed by Ch.~Hermite in 1848 (see e.g.~\cite{Picard1901}, \cite{Hermite1850}) in a very broad form.
In this paper we follow C.G.J.~Jacobi and O.~Perron and study algebraic periodicity
of generalized Euclidean algorithms in $\r^3$.

\vspace{2mm}

The first generalization of the classical Euclidean algorithm to the higher dimensional case
was introduced  in 1868 by C.G.J.~Jacobi  in~\cite{Jacobi1868} and further developed in~\cite{Perron1907}
and~\cite{Perron1913} by O.~Perron.
Let us outline the algorithm.

\vspace{2mm}

%

\myexample{ Jacobi-Perron algorithm}
{
{\noindent{\bf \underline{Input of the algorithm}:}}
Any triples of real numbers $(x_0,y_0,z_0)$.

\vspace{1mm}

{\noindent {\bf \underline{Step of the algorithm}:}}
Assume that we have constructed a vector $(x_i, y_i,z_i)$ and let $y_i\ne 0$.
Then we define the next vector $(x_{i+1},y_{i+1},z_{i+1})$ as follows
$$
(x_{i+1},y_{i+1},z_{i+1})=\Big(y_i,z_i-\Big\lfloor\frac{z_i}{y_i}\Big\rfloor y_i, x_i-\Big\lfloor\frac{x_i}{y_i}\Big\rfloor y_i\Big).
$$
As an output of this step we have the following pair of integers:
$$
\Big(\Big\lfloor\frac{z_i}{y_i}\Big\rfloor,\Big\lfloor\frac{x_i}{y_i}\Big\rfloor\Big).
$$
It is called the {\it $i$-th element} of the Jacobi-Perron multidimensional continued fraction algorithm.

\vspace{1mm}

{\noindent{\bf \underline{Termination of the algorithm}:}}
Once we have constructed a vector  $(x_i, y_i,z_i)$ satisfying $y_i=0$,
the algorithm terminates.

}

\vspace{2mm}

If one applies the Jacobi-Perron algorithm to cubic vectors, it never terminates and sometimes produces an eventually
periodic output.  So it is natural to ask the following question: {\it Is the output of the Jacobi-Perron algorithm
always eventually periodic for cubic vectors?}
In folklore this question is known as Jacobi's Last Theorem (for further details see, e.g.,
in~\cite{Schweiger2000} and Chapter~23.4 in~\cite{KarpenkovGCF2013}).

\begin{problem}{\bf (Jacobi's Last Theorem.)}
Let $K$ be a totally-real cubic number field.
Consider arbitrary elements $y$ and  $z$ of $K$
satisfying $0<y,z<1$ such that $1$, $y$, and $z$ are independent over $\q$.
Is it true that the Jacobi-Perron algorithms generate an eventually periodic continued fraction with
starting data $v=(1,y,z)$?
\end{problem}

This problem is still open and the answer to it is conjectured to be negative.
Numerical computations in~\cite{Elsner1967}
show that the output is not periodic, for instance, for  the vector
$v=(1,\sqrt[3]{4},\sqrt[3]{16})$.
Periodicity of the Jacobi-Perron algorithm was proven for certain classes of cubic numbers
in~\cite{Bernstein1971}.
It is interesting to notice that cubic periodicity is also unknown for various modifications
of the Jacobi-Perron algorithm
(subtractive algorithms by V.~Brun~\cite{Brun1958} and
E.S.~Selmer~\cite{Selmer1961},
fully subtractive algorithm by F.~Schweiger~\cite{Schweiger1994}
and~\cite{Schweiger1995},
generalized subtractive algorithm~\cite{Schweiger1992},
 Tamura-Yasutomi algorithm~\cite{Tamura2009},
 heuristic algebraic  periodicity detecting algorithm~\cite{Karpenkov2021-2},  etc.).

\vspace{2mm}

Let us try to explain informally the reason why
periodicity may break for cubic vectors.
Any cubic vector is an eigenvector of some $3\times 3$-matrix with
integer coefficients whose characteristic polynomial is irreducible over $\q$.
Due to Dirichlet's unit theorem (see, e.g., in~\cite{Borevich1966})
this matrix can be taken to have a unit determinant  (see, e.g., in Chapter 17 of~\cite{KarpenkovGCF2013}).
This matrix has another two (possibly complex) eigenlines.
We say that  these eigenlines are {\it conjugate} to the line containing the original cubic vector.
The algorithms discussed before work entirely with the direction of the original vector
disregarding the directions of the conjugate eigenlines.
This is probably the main reason for the loss of periodicity.
The study of Klein's polyhedra suggests that all three eigenlines should be considered in a periodic algorithm.
In fact the situation here is quite paradoxical, as conjugate eigenlines can be constructed from the original vector
(see Subsection~\ref{2ways}).
We discuss Klein's polyhedra later in Subsection~\ref{Klein sails and three-dimensional continued fractions}
as they are substantially used in the proof of periodicity for the new $\sin^2$-algorithm.

\vspace{2mm}

The main goal of this paper is to develop a new modification
of the Jacobi-Perron algorithm which we call the $\sin^2$-algorithm  (see Section~\ref{sin-Jacobi-Perron algorithm}).
The $\sin^2$-algorithm works with triples of arbitrary real vectors in $\r^3$.
For triples of totally-real conjugate vectors we prove periodicity
of the  $\sin^2$-algorithm (Theorems~\ref{eventually-periodic}
and~\ref{eventually-periodic-algorithm}).
This provides a solution to Hermite's problem in the spirit of Jacobi's Last Theorem for
cubic vectors in the totally-real case of $\r^3$.
To the best of our knowledge this is the first algorithm that has been proven to be periodic if and only if
it is applied to cubic vectors.
The non-totally-real case remains now open; we briefly discuss it  in Section~\ref{conjectures}.

\vspace{2mm}

Finally let us mention that various different types of cubic periodicity were also studied for other types of generalized continued fractions:
for Klein polyhedra~\cite{Lachaud1993,German2008},
Minkovski-Voronoi polyhedra~\cite{Voronoi1952,Minkowski1967,Bullig1940},
triangle sequences~\cite{Dasaratha2012}, and
ternary continued fractions (or bifurcating continued fractions)~\cite{Murru2015}.

\vspace{2mm}

{\noindent
{\bf This paper is organized as follows.}
}
We start in Section~\ref{Basic notions and definitions}
with basic notions and definitions.
Here we discuss the concept of generalized Euclidean algorithms; recall some notions of integer geometry;
show a classic geometric construction of Klein's polyhedra that are used in the proofs;
and discuss the action of the positive Dirichlet group on cubic sails.
In Section~\ref{sin-Jacobi-Perron algorithm}
we introduce the $\sin^2$-algorithm and state its periodicity (Theorems~\ref{eventually-periodic}
and~\ref{eventually-periodic-algorithm}).
In Section~\ref{Lattice geometry used in the proofs}
we show several important statements from geometry of integer lattices
and deduce the proof of Theorem~\ref{eventually-periodic}$($i$)$.
Further in Section~\ref{Theorem-ii}
we reformulate the second item of Theorem~\ref{eventually-periodic}
in a more analytic way (see Theorem~\ref{greater e}).
Finally we prove all items of Theorem~\ref{greater e}
in Sections~\ref{teor-v}, \ref{teor-i}, \ref{teor-ii}, and~\ref{theor-iii-iv}.
Note that some computations in the proof of Theorem~\ref{greater e}
are done in MAPLE2020 (see in~\cite{maple}).
We conclude this paper in Section~\ref{conjectures}
with a list of open questions.

\section{Definition of the $\sin^2$-algorithm and theorem on its periodicity}
\label{sin-Jacobi-Perron algorithm}

We start in Subsections~\ref{2ways} and~\ref{States of the main part of the algorithm}
with  a general discussion on conjugate cubic directions.
In Subsection~\ref{Technique to construct separating bases for arbitrary states}
we introduce the notion of separating bases and write down the algorithm to reach a
separating basis starting from an arbitrary integer basis.
We define admissible JP-transformations and discuss their elementary properties
in Subsection~\ref{Admissible JP-transformations}.
We continue in Subsection~\ref{Definition of the map}
by showing how to pick an admissible JP-transformation
for the next iteration of the $\sin^2$-algorithm.
In Subsection~\ref{Description ALG} we write down
the $\sin^2$-algorithm and go through one particular example.
Finally in Subsection~\ref{Periodicity of Phi}
we state the periodicity of the algorithm (Theorem~\ref{eventually-periodic-algorithm}).

\subsection{How to set conjugate cubic directions in $\r^3$}\label{2ways}

Let us now briefly discuss two ways to define cubic conjugate directions
in $\r^3$.

\vspace{2mm}

{\noindent {\bf Polynomial definition.}}
Let  $p$ be an irreducible over $\q$ cubic polynomial with integer coefficients.
Then its roots $\xi$, $\nu$, and $\mu$ are called {\it conjugate} cubic numbers.

\vspace{2mm}

In order to define conjugate vectors we pick a basis
$(q_1,q_2,q_3)$
in the linear space of degree 2 polynomials of over $\q$  in one variable.
Then the vectors
$$
(q_1(\xi),q_2(\xi),q_3(\xi)), \quad
(q_1(\nu),q_2(\nu),q_3(\nu)), \quad
(q_1(\mu),q_2(\mu),q_3(\mu))
$$
are called {\it conjugate} cubic vectors.
Once we know 4 polynomials: $(p,q_1,q_2,q_3)$
we uniquely determine a triple of cubic conjugate vectors.

\begin{remark}\label{propor}
Note that all Jacobi-Perron type algorithms generate the same output for proportional vectors.
For that reason we are interested only in the directions of  conjugate vectors:
we call such directions  {\it conjugate cubic directions}.
\end{remark}

In order to define three conjugate cubic directions it is enough to consider polynomials
$(p,q_1,q_2,q_3=1)$ (where  as before $1$, $q_1$, and $q_2$ generate the linear space of polynomials of degree 2).

\begin{remark}
Note that conjugate cubic vectors and directions can be considered in both totally-real and non-totally-real cases.
\end{remark}

{\noindent {\bf Matrix definition.}}
An alternative approach to define conjugate cubic directions is
to consider  a triple of eigendirections for some integer $3\times 3$-matrix.
The characteristic polynomial of such matrices should be irreducible over $\q$.

\begin{remark}\label{cd}
Polynomial and matrix definitions are equivalent in the following sense.
They provide the same set of triples of conjugate cubic directions.
Furthermore, any triple of conjugate cubic directions can be defined by an $SL(3,\z)$-matrix
(see, e.g., Chapter 17 in~\cite{KarpenkovGCF2013}).
\end{remark}

\begin{remark}
Notice that similar constructions work for
algebraic conjugate degree $d$ directions in $\r^d$.
\end{remark}

\subsection{Totally-real cubic states}
\label{States of the main part of the algorithm}
For simplicity we further use the following notion.

\begin{definition}\label{domain of alg}
A {\it state} is collection of coordinates for three linearly independent vectors $(\xi=(x,y,z),\nu_1,\nu_2)$
where
$$
x\ge y\ge z>0.
$$
Denote the set of all states by $\Sigma(\xi,\nu_1,\nu_2)$.
\end{definition}

In this paper we are mostly  interested in the following totally-real cubic states.

\begin{definition}
We say that a state $s=(\xi,\nu_1,\nu_2)$ is a {\it totally-real cubic} state if
there exists a matrix  $B\in\SL(3,\z)$ with irreducible characteristic polynomial over $\q$ such that the vectors $\xi,\nu_1,\nu_2$ are eigenvectors of $B$ with distinct eigenvalues.
\end{definition}

\begin{remark}
Note that a totally-real cubic state $s=((x_0,y_0,1),(x_1,y_1,1),(x_2,y_2,1))$
is determined by the first vector $(x_0,y_0,1)$ in a unique way up to a swap of the last two vectors.
\end{remark}

\subsection{Technique to construct separating bases for arbitrary states}
\label{Technique to construct separating bases for arbitrary states}

Let us start with several general definitions. A vector is called {\it integer} if all its coordinates are integers.
We say that a basis is {\it integer}
if it generates the lattice of integer points.
Now we introduce a couple of notions that are used in the main algorithm.

\begin{definition}\label{separating-def}
Let $(\xi, \nu_1,\nu_2)$ be a triple of vectors.

\begin{itemize}
\item
A basis $E$ is called {\it supporting} for $\xi$
if the non-negative orthant of $E$ contains either $\xi$ or $-\xi$

\item
A supporting basis for $\xi$ is said to be {\it separating} for $(\xi, \nu_1,\nu_2)$
if  a non-negative orthant does not include both $\pm\nu_1$ and $\pm\nu_2$.

\item
We say that a state $s$ is {\it separating} if the coordinate basis is separating for the vectors of $s$.
\end{itemize}
\end{definition}

Let us now show how to find a separating basis for a given state.

\vspace{2mm}

%

\myexample{Algorithm producing a separating basis}
{
{\noindent{\bf \underline{Input of the algorithm}:}}
We start with three linearly independent vectors $(\xi,\nu_1, \nu_2)$
where the coordinates of $\xi$ are positive in the coordinate basis $E$.

\vspace{1mm}

{\noindent{\bf \underline{Set up for the algorithm}:}}
Let
$$
\frac{x}{y}=[a_0:a_1;a_2;\ldots]
$$
be a regular continued fraction and let $\frac{p_n}{q_n}$
be its $n$-th partial quotient, i.e.,
$$
\frac{p_n}{q_n}=[a_0:a_1;a_2;\ldots; a_n]
$$
(with positive relatively prime integers $p_n, q_n$, and $n=0,1,2,\ldots$).

Denote by $M_{2n,\xi}(E)$ the following basis
$$
(p_{2n+1}e_1+ q_{2n+1}e_2 , p_{2n} e_1+q_{2n} e_2,e_3).
$$
Namely, the matrix of transformation to $M_{2n,\xi}(E)$ is expressed as follows:
$$
M_{n,\xi}=\prod\limits_{k=0}^{n}
\left(
\left(
\begin {array}{ccc}
1&a_{2k}&0\\
\noalign{\medskip}0&1&0\\
\noalign{\medskip}0&0&1
\end {array}
\right)
\cdot
\left(
\begin {array}{ccc}
1&0&0\\
\noalign{\medskip}a_{2k+1}&1&0\\
\noalign{\medskip}0&0&1
\end {array}
\right)
\right)
$$

{\noindent {\bf \underline{Step of the algorithm}:}}
At Step~$i$ we construct the basis $M_{2n,\xi}(E)$ and check
whether it  is a separating for $s$.

\vspace{1mm}

{\noindent{\bf \underline{Termination of the algorithm}:}}
The algorithm terminates if $M_{2i,\xi}(E)$ is a separating basis
or if the continued fraction for $x/y$ is finite and we have reached the last element.
}

\vspace{2mm}

\begin{proposition}\label{ConstSetB}
Let $s=(\xi,\nu_1,\nu_2)$ be a totally-real cubic state.
Then the above algorithm terminates and produces a separating basis.
\end{proposition}

\begin{proof}
Let us outline the proof.
From the classical theory of regular continued fractions we have the following statements.
For every $n\ge 1$
\begin{itemize}
\item  the positive octant for  $M_{n,\xi}(E)$ contains $\xi$;

\item the vectors of $M_{n,\xi}(E)$ generate $\z^3$;

\item the orientation of $M_{n,\xi}(E)$ coincides with the orientation with $E$;

\item For partial quotients we have
$$
\lim\limits_{n\to \infty}\frac{p_n}{q_n}=\frac{x}{y}.
$$
\end{itemize}

From the general theory of algebraic numbers the vectors $e_3$, $\xi$ and $\nu_1$
(and also the vectors $e_3$, $\xi$ and $\nu_2$) are linearly independent. Hence
the projections of $\nu_1$ and $\nu_2$ along $e_3$ to the plane $z=0$
are not collinear to the projection of $\xi$ along $e_3$ to the plane $z=0$.
Thus there exists $n>0$ such that the cone generated by $M_{n,\xi}(E)$
is sufficiently close to the plane passing through 1 and $\xi$, and, therefore, it does not contain both $\nu_1$
and $\nu_2$.
\end{proof}

 \subsection{Admissible JP-transformations}
 \label{Admissible JP-transformations}

First of all let us give a general definition of admissible transformations.

\begin{definition}
Let $E$ be a supporting basis for a vector $\xi$.
We say that a transformation (or simply a matrix) $M$ is  {\it admissible} for the pair $(\xi, E)$  if
$M(E)$ is a supporting basis for~$\xi$.
\end{definition}

Similar to the Jacobi-Perron algorithm, the $\sin^2$-algorithm operates with a collection of linear transformations.
This collection contains transformations of two types.

\begin{definition}
Let $E=(e_1,e_2,e_3)$ be some basis.
\begin{itemize}
\item
Denote by
$V_{\alpha,\beta;\gamma}(E)$ the basis
$$
(e_1,e_2+\gamma e_1,e_3+\alpha e_1+\beta e_2).
$$

\item
Denote by
$W(E)$ the basis
$$
(e_1+e_3,e_2+e_1, e_3).
$$
\end{itemize}
We call these basis transformations {\it JP-transformations}.
\end{definition}

Let us now introduce the notion of the set of admissible maps
for the $\sin^2$-algorithm.

\begin{definition}\label{admissible}
Consider a vector $\xi$ with a supporting basis $E$.
Denote by $\M_\xi$ the union of
\begin{itemize}
\item the set of all admissible JP-transformations $V_{\alpha,\beta;\gamma}$;
\item the transformation $W$ in case if $W(E)$ is supporting for $\xi$.
\end{itemize}
The elements of $\M_\xi$ are {\it admissible JP-transformations}.
We say that $\M_\xi$ is the {\it set of admissible JP-transformations} for $\xi$.
\end{definition}

Let us collect some basic properties of JP-transformations together.

\begin{proposition}
Let $\xi$ be a non-zero vector with a supporting basis $E$,
such that the coordinates $(x,y,z)$ of $\xi$
in the basis $E$ satisfy $x>y>z$.
The the following statements hold.
\\
{\it i$)$} The JP-transformation $V_{\alpha,\beta,\gamma}$ is admissible if and only if
$$
\alpha\le \Big\lfloor \frac{x}{z}\Big\rfloor, \quad
\beta \le \Big\lfloor \frac{y}{z}\Big\rfloor,  \quad \hbox{and} \quad
\gamma \le \Big\lfloor \frac{x/z-\alpha}{y/z-\beta}\Big\rfloor
$$
{\it ii$)$} The JP-transformation $W(E)$  is admissible if and only if
$$
z>x-y>0.
$$
\\
{\it iii$)$} Every JP-transformation is an injection of the positive orthant of $E$ to itself.
\\
{\it iv$)$} Let $E$ be a separating basis for a triple of vectors $(\xi,\nu_1,\nu_2)$
and let $M$ be any admissible JP-transformation for $\xi$.
Then $M(E)$ is a separating basis for $(\xi,\nu_1,\nu_2)$ .
\end{proposition}

\begin{proof}
Let us first  write the coordinates of $\xi$ in the new bases for $V_{\alpha, \beta, \gamma}$ and $W$, we have:
$$
\begin{array}{l}
V_{\alpha,\beta;\gamma}(x,y,z)=\big(x-\alpha z-\gamma(y-\beta z), y-\beta z, z);\\
W(E)=\big(x - y, y, z-(x-y)\big).
\end{array}
$$
All these coordinates should be positive; this implies the inequalities of first two items of the proposition.

\vspace{2mm}

The third item is  straightforward.
The last item is as follows.
Since $E$ is a separating basis, the cone spanned
by $E$ does not include both $\pm\nu_1$ and $\pm\nu_2$.
Therefore, by the third item the cone spanned by $M(E)$
does not include both $\pm\nu_1$ and $\pm\nu_2$.
\end{proof}

\begin{remark}
Note that for every $\xi$ the set $\M_\xi$ is finite.
\end{remark}

\subsection{Definition of the map $\Phi$}
\label{Definition of the map}

The main idea of the $\sin^2$-algorithm aims at to maximize at every step the
$\sin^2$-function of the following angle.

\begin{definition}
Consider a state $s=(\xi,\nu_1,\nu_2)$.
Denote by $\alpha(s)$ the angle between the planes through the origin spanned by
pairs of vectors
$(\xi,\nu_1)$ and the $(\xi,\nu_2)$ respectively.
\end{definition}

The optimization rule for the $\sin^2$-algorithm is given by the following definition.

\begin{definition}\label{optimization rule-def}
Let $s$ be a state.
Let  $M\in \M_\xi$ be
the transformation with the greatest possible value of $\sin\alpha(M(s))$.
Let
$$
\Phi(s)=T\circ M,
$$
where $T$ is a transposition of the basis vectors that puts $M(\xi)$
in decreasing order.
We call the transformation $\Phi(s)$ the {\it $\sin^2$-transformation} for $s$.
\end{definition}

\begin{remark}
The projectivisation of $\sin^2$-transformations gives a two-dimen\-si\-onal analog
of the Gauss map for classical continued fractions.
\end{remark}

\begin{remark}
For the case of cubic totally-real states
the maximum of $\sin\alpha(M(s))$ is uniquely defined.
However in the case of arbitrary vectors
several maxima are possible.
In this case one should introduce an ordering for the set $\M_\xi$
and take the first element of $\M_\xi$ providing the maximum.
One can take a standard {\it lexicographic}
ordering for the elements $V_{*,*;*}$
and additionally setting $W>V_{*,*;*}$ if $W\in \M_\xi$.
\end{remark}

\subsection{Description of the $\sin^2$-algorithm}
\label{Description ALG}

Now everything is ready for the formulation of the $\sin^2$-algorithm.

\vspace{2mm}

%

\myexample{$\sin^2$-algorithm}
{
{\noindent
{\bf \underline{Input data}:} We are given by three linearly independent vectors $\xi$, $\nu_1$, $\nu_2$ in $\r^3$.
We start with the coordinate basis $E$.
}
\\
{\bf \underline{Preliminary Stage 1: Finding a supporting basis}.}
First we multiply the vectors of $E$ by $\pm 1$ and swap them such that the coordinates
of the vector $\xi$ in the new basis satisfy
$$
x\ge y\ge z>0.
$$

\vspace{2mm}

{\noindent
{\bf \underline{Preliminary Stage 2: Finding a separating basis}.}
We use the techniques of Proposition~\ref{ConstSetB} to construct a separating basis for
$(\xi,\nu_1, \nu_2)$.
In these coordinates the triple of vectors will be a separating state.
}

\vspace{1mm}

{\noindent
{\bf \underline{Main Stage}:}
Each step we perform the following iteration: $s \to \Phi(s)$.
}

{\noindent
{\bf \underline{Termination of the algorithm}:}
Once we produce a triple $(z,y,z)$ with one of the coordinates equal to 0,
the algorithm terminates.
}

\vspace{1mm}

{\noindent
{\bf \underline{Output}:}
A sequence of admissible JP-transformations generated in the iterations of
Preliminary Stages~1 and~2 and Main Stage.
}

}

\begin{remark}
In case if some coordinates of $\xi$ coincide
there is some freedom to choose a basis transposition that put the coordinates of $\xi$ in the non-increasing order.
Let us agree not to swap the order of the equal coordinates with respect each other.
Note that the coordinates of cubic vectors are always distinct, hence the swap of coordinates
are uniquely defined.
\end{remark}

\begin{remark}
For a MAPLE2020 realisation of the $\sin^2$-algorithm we refer to Sin2JP.mw in~\cite{maple}.
\end{remark}

\begin{example}
Let us consider three linearly-independent eigenvectors $\xi$, $\nu_1$, and $\nu_2$ of
the matrix
$$
\left(
\begin {array}{ccc}
0 & 0 & 1 \\
1 & -15 & -9 \\
-9 & 136 & 66
\end {array}
\right)
$$
We pick $\xi(x,y,1)$ to be the eigenvector with the greatest eigenvalue.
Note that
$$
\xi\approx(0.02189094967, -0.1479558970, 1).
$$

\vspace{2mm}

{\noindent
At Preliminary Stage~1} we swap the first and the third coordinate vector and multiply the second one by $-1$.
Namely we consider the lattice basis transformation with the following matrix:
$$
\left(
\begin {array}{ccc}
0&0&1\\
0&-1&0\\
1&0&0
\end {array}
\right).
$$
The new basis is supporting for $\xi$.

\vspace{2mm}

{\noindent
At Preliminary Stage~2} we consider the regular continued fraction for $x/y$ which is
$$
\frac{x}{y}=[6;1:3:6:\ldots].
$$
We arrive to a separating basis after we use the first two elements of the continued
fraction for $x/y$ (which are 6 and 1).
The matrix of the corresponding basis transformation is as follows:
$$
\left(
\begin {array}{ccc}
1&{\bf 6}&0\\
0&1&0\\
0&0&1
\end {array}
\right)
\cdot
\left(
\begin {array}{ccc}
1&0&0\\
{\bf 1}&1&0\\
0&0&1
\end {array}
\right).
$$

{\noindent During the main stage} we get a periodic sequence of the following integer basis linear transformations.
Its pre-period consists of three steps:
$$
\Phi_1=
\left(
\begin{array}{ccc} 1&3&1\\
\noalign{\medskip}0&1&1\\
\noalign{\medskip}0&1&0
\end{array} \right),
\quad
\Phi_2=
\left(
  \begin{array}
{ccc} 1&1&1\\ \noalign{\medskip}1&0&1\\ \noalign{\medskip}1&0&0
\end{array} \right) ,
\quad
\Phi_3=
\left( \begin{array}{ccc} 1&1&0
\\ \noalign{\medskip}1&0&1\\ \noalign{\medskip}1&0&0\end{array}
 \right).
$$
and its period is

$$
\begin{array}{ccc}
\Phi_{8k+4}=
\left( \begin{array}{ccc} 1&1&0\\ \noalign{\medskip}0&1&1
\\ \noalign{\medskip}0&0&1\end {array} \right) ,
&
\Phi_{8k+5}=
\left( \begin{array}
{ccc} 1&0&1\\ \noalign{\medskip}1&0&0\\ \noalign{\medskip}0&1&0
\end {array} \right) ,
&
\Phi_{8k+6}=
\left( \begin{array}{ccc} 1&0&1
\\ \noalign{\medskip}1&0&0\\ \noalign{\medskip}0&1&0
\end{array}\right),
\\
\Phi_{8k+7}=
 \left( \begin{array}{ccc} 1&1&6\\ \noalign{\medskip}1&0&5
\\ \noalign{\medskip}0&0&1\end{array} \right),
&
\Phi_{8k+8}=
\left( \begin{array}{ccc} 1&16&4\\ \noalign{\medskip}0&4&1\\ \noalign{\medskip}0&1&0
\end {array} \right) ,
&
\Phi_{8k+9}=
\left( \begin{array}{ccc} 1&0&1
\\ \noalign{\medskip}1&0&0\\ \noalign{\medskip}0&1&0\end {array}
 \right) ,
\\
\Phi_{8k+10}=
\left( \begin{array}{ccc} 0&1&1\\ \noalign{\medskip}1&0&0
\\ \noalign{\medskip}0&0&1\end{array} \right) ,
&
\Phi_{8k+11}=
\left( \begin{array}{ccc} 1&1&1\\ \noalign{\medskip}1&1&0\\ \noalign{\medskip}1&0&0
\end{array} \right),
\end{array}
$$
for $k=0,1,\ldots$.

A MAPLE2020 realisation of this example is in Sin2JP.mw in~\cite{maple}.
\end{example}

\subsection{Periodicity of the $\sin^2$-algorithm for triples of real conjugate cubic vectors}
\label{Periodicity of Phi}

We start with some technical definitions that is further used in the proofs.

\begin{definition}
Let $\xi$ be a cubic totally-real vector and let
$\nu_1$ and $\nu_2$ be its conjugate vectors.
\begin{itemize}
\item
We say that a state $s=((x_0,y_0,1),(x_1,y_1,1),(x_2,y_2,1))$ is a $\xi$-state if
in some $\z^3$-basis the coordinates of vectors $\xi$, $\nu_1$ and $\nu_2$ are respectively proportional to the vectors of $s$.

\item
The set of all separating $\xi$-states $s$ where $\sin\alpha(s)>\varepsilon$ is said to be
{\it $\varepsilon$-cap} for $(\xi,\nu_1,\nu_2)$ and denoted by $\Omega_\varepsilon(\xi;\nu_1,\nu_2)$.

\item
The set of all separating $\xi$-states $s$ satisfying
$$
\sin^2\alpha(s)>\sin^2\alpha(\Phi(s))
$$
is said to be
{\it extremal} for $\xi$ and denoted by $\Omega_{\max}(\xi;\nu_1,\nu_2)$.
\end{itemize}
\end{definition}

\begin{remark}
In fact, the angle between planes is always in $[0,\pi/2]$, and hence $\sin\alpha$ is always nonnegative.
So considering $\sin\alpha$ here is equivalent to considering $\sin^2\alpha$.
We pick $\sin^2\alpha$ as it has a nice rational expressions in terms of state coordinates.
\end{remark}

The main result of current paper is based
on the statements of the following theorem together with its corollary.

\begin{theorem}\label{eventually-periodic}
Let $A\in GL(3,\z)$ be a matrix with irreducible over $\q$ characteristic polynomial
having three real roots.
Let also $(\xi,\nu_1,\nu_2)$ be some eigenbasis of $A$.
Assume that the coordinate basis is separating.
Then the following two statements hold.

{\noindent
{\it $($i$)$}}
For every $\varepsilon>0$ the set $\Omega_\varepsilon(\xi;\nu_1,\nu_2)$ is finite.

{\noindent
{\it $($ii$)$}}
The set $\Omega_{\max}(\xi;\nu_1,\nu_2)$ is finite.

\end{theorem}

\begin{proof}
Item~$($i$)$ is proven in Subsection~\ref{Theorem-i}.
Item~$($ii$)$ is shown in Section~\ref{Theorem-ii}.
\end{proof}

\begin{corollary}\label{period-c1}
Let $A\in GL(3,\z)$ be a matrix with irreducible over $\q$ characteristic polynomial
having three real roots.
Let $(\xi,\nu_1,\nu_2)$ be its eigenvalues.
Suppose that the coordinate basis is separating.
Then sequence of transformations $\Phi^n(\xi;\nu_1,\nu_2)$ is a eventually periodic.
\end{corollary}

\begin{proof}
First of all, note that $\Phi(\xi;\nu_1,\nu_2)$ preserves the property of a basis to be separating.
Since $\Omega_\varepsilon(\xi;\nu_1,\nu_2)$ is finite for every $\varepsilon>0$ (Theorem~\ref{eventually-periodic}$($i$)$),
we reach the set $\Omega_{\max}(\xi;\nu_1,\nu_2)$ in finitely many steps.

Each time we reach $\Omega_{\max}(\xi;\nu_1,\nu_2)$ we ``descend'', i.e., reduce the value of $\sin\alpha$ on the next step.
Since $\Omega_{\max}(\xi;\nu_1,\nu_2)$ is finite (Theorem~\ref{eventually-periodic}$($ii$)$), we have only many descend steps.

Therefore, we have periodicity of the output as our algorithm is deterministic.
\end{proof}

\begin{definition}
We say that three vectors in $\r^3$ are in {\it $\z$-general position}
if any two of them spans a plane that does not have non-zero integer vectors in it.
\end{definition}

\begin{remark}
Note that triples of cubic conjugate vectors are always in $\z$-general position.
\end{remark}

Finally the main result of this  paper can be formulate as follows.

\begin{theorem}\label{eventually-periodic-algorithm}
Consider a triple $(\xi,\nu_1,\nu_2)$ of linearly independent vectors in $\z$-general position in $\r^3$.
Then the $\sin^2$-algorithm is eventually periodic for $(\xi,\nu_1,\nu_2)$
if and only if $(\xi,\nu_1,\nu_2)$ is a triple of conjugate cubic vectors.
\end{theorem}

\begin{proof}
After the preliminary stages we arrive to a separating basis for $(\xi,\nu_1,\nu_2)$.
The finiteness of steps in the  preliminary stages follows from $\z$-general position for $(\xi,\nu_1,\nu_2)$.
By Dirichlet unit's theorem,
there exists a $GL(3,\z)$-matrix with irreducible over $\q$ characteristic polynomial
having three real roots
(see, e.g., in Chapter 17~\cite{KarpenkovGCF2013}).
Then by Corollary~\ref{period-c1}
the algorithm is eventually periodic.

\vspace{2mm}

The converse statement is straightforward.
Let  the $\sin^2$-algorithm be eventually periodic for $(\xi,\nu_1,\nu_2)$.
Since the triple $(\xi,\nu_1,\nu_2)$  is in $\z$-general position,
we reach the main stage in finitely many steps of the preliminary stages.
Let $M_1$ and $M_2$ be the products of all the matrices for the transformations
in the pre-period and the period respectively.
Then the matrix
$$
M=M_1M_2M_1^{-1}
$$
will have the triple $(\xi,\nu_1,\nu_2)$ as an eigenbasis.

It remains to check that the characteristic polynomial of $M$ is irreducible over $\q$.
Assume that the characteristic polynomial of $M$ is reducible.
Then $M$ has a  rational eignevector.
Now if $M$ has more than 1 distinct eigenvalues, then the irreducibility
follows directly from the fact that the triple $(\xi,\nu_1,\nu_2)$  is in $\z$-general position.
Hence $M$ has to be proportional to the identity matrix $I$, and, therefore, $M_2$
is proportional to $I$ as well.
Finally, by construction, $M_2$ maps the positive octant to
a proper subset of the positive octant. Hence $M_2$ is not multiple to $I$.
We come to the contradiction. Therefore,
the characteristic polynomial of $M$ is irreducible over $\q$.
\end{proof}

\begin{remark}
If we remove the condition of $\z$-general position,
then for some triples the algorithm of the preliminary stage may generate an infinite
output, which can be periodic if the corresponding continued fraction is periodic.
In fact this issue can be resolved by modifying the algorithm of the preliminary step.
Also as we conjecture below (see Conjecture~\ref{conj-2}),
one might skip the preliminary stage of the algorithm.
We omit further discussions on vectors that are not in $\z$-general position
as this case is not of interest for us.
\end{remark}

\section{Cubic Klein's polyhedra and corresponding Dirichlet groups}
\label{Basic notions and definitions}

In this section we collect some preliminary notions and definitions.
In Subsection~\ref{Integer Geometry} we briefly recall some basic notions from integer geometry.
Further in Subsection~\ref{Klein sails and three-dimensional continued fractions}
we introduce Klein's sails, which is a very useful geometrical generalization of
regular continued fractions to multidimensional case in terms of polyhedral surfaces.
In Subsection~\ref{Positive Dirichlet group and its action on cubic sails}
we discuss the notion of positive Dirichlet groups and its relation
to double-periodicity of Klein's sails.
We conclude in Subsection~\ref{2ways}
with a brief discussion on how to define an algebraic directions in order to generate the input of the algorithm.

\subsection{Basic notions of integer geometry}\label{Integer Geometry}

Let us start with several notions of integer geometry.
(see~\cite{KarpenkovGCF2013} for further reference).

\vspace{2mm}

A point is called {\it integer} if all its coordinates are integers.
A segment (or a vector) is called {\it integer}
if the coordinates of its endpoints are integer.
A {\it polygon} is called integer if its vertices are all integer.
Finally, a plane (a line) is {\it integer} if its intersection with $\z^3$
is a two-dimensional lattice (respectively, a one-dimensional lattice).

\vspace{2mm}

Further we use one important invariant of integer geometry: the integer distance.

\begin{definition}
Let $\pi$ be an integer plane and let $A$ be an integer  point in the complement to the plane $\pi$.
The index of the sublattice generated by all vectors with one endpoint being $A$ and the other contained in $\pi$
in the lattice of all integer vectors of the span of $A$ and $\pi$ is called
the {\it integer distance} between $A$ and $\pi$.
It is denoted by $\id(A,\pi)$.
\end{definition}

\begin{remark}
Note that the integer distance from a point $A$ to a plane $\pi$
coincides with  the number of integer planes parallel to $\pi$ that are
between $A$ and $\pi$.
Let $B$ be a point of $\pi$.
Then such planes split the segment $AB$ into $k$ segments of equal length, where $k$
is equal to the integer distance between  $A$ and $\pi$.
\end{remark}

Let us discuss also the notion of integer volume and its relation to the Euclidean volume.

\begin{remark}
By definition, the {\it integer volume} of an integer tetrahedron $T$
is the index of the lattice generated by its edges in the lattice $\z^3$.

\vspace{1mm}

{\noindent
Note that the integer volume of $T$ equals to $1/6$ of the Euclidean volume of $T$.
By that reason we will use a more usual notion of the Euclidean volume
keeping in mind its integer invariance. }

\vspace{1mm}

{\noindent
We denote the Euclidean volume of a polytope $P$ by $\vol(P)$.}
\end{remark}

\subsection{Klein's sails and three-dimensional continued fractions}
\label{Klein sails and three-dimensional continued fractions}
Later in the proofs we substantially use a geometric construction of polyhedral multidimensional continued fractions  that was introduced by F.~Klein in~\cite{Klein1895} and~\cite{Klein1896}. (For general information on Klein polyhedra we refer to~\cite{Arnold2002} and~\cite{KarpenkovGCF2013}.)

\vspace{2mm}

First of all we recall that a {\it simplicial cone} in $\r^3$ is the convex hull of $3$ rays with the same vertex and linearly
independent directions.

\begin{definition}
Consider an integer simplicial three-dimensional cone $C$ with center at the origin in $\r^3$.
\begin{itemize}
\item
An {\it A-hull} of the cone $C$ is the convex
hull of all integer points lying in the closure of $C$ except the origin.
We denote this set by $\Ahull(C)$.

\item
The boundary of the set $\Ahull(C)$ is called
the {\it sail} of this cone.
Denote it by $K(C)$.
\end{itemize}
\end{definition}

\begin{definition}
Consider three planes of $\r^{3}$ intersecting at a single point at the origin.
The complement to the union of these planes consists of $8$ open simplicial
cones $C_1,\ldots, C_8$.
The set of all sails for all the cones $C_1,\ldots,C_8$
is called the {\it three-dimensional continued fraction} ({\it in the sense of Klein})
associated to the given three planes in $\r^3$.
\end{definition}

\subsection{Positive Dirichlet group and its action on cubic sails}
\label{Positive Dirichlet group and its action on cubic sails}

Positive Dirichlet groups are very natural groups of matrices acting on sails of cubic cones.
Their presence is the reason for the double periodicity of Klein's sails.
In this subsection we give the main definitions related to Dirichlet groups.

\vspace{2mm}

Consider a matrix $A\in\GL(n,\z)$ whose characteristic
polynomial is irreducible over $\q$.

\begin{definition}
By $\Gamma(A)$ we denote the set of all integer matrices commuting with $A$.

\vspace{1mm}

\begin{itemize}
\item
The {\it Dirichlet group} $\Xi(A)$ is the subset of invertible matrices in $\Gamma(A)$.

\vspace{1mm}

\item
The {\it positive Dirichlet group} $\Xi_+ (A)$ is the subgroup of $\Xi(A)$
containing of all matrices with positive real eigenvalues.
\end{itemize}
\end{definition}

In case of $A\in \GL(3,\z)$ with distinct real (irrational) eigenvalues,
the group  $\Xi_+ (A)$ is isomorphic to $\z^{2}$.

\begin{definition}
Consider $A\in\GL(3,\z)$ whose eigenvalues are all distinct and real.
Take three planes passing through the origin
that are spanned by pairs of linearly independent eigenvectors of $A$.
These planes define the three-dimensional continued fraction that is called the
{\it multidimensional continued fraction associated to $A$}.
The sails of the corresponding cones are said to be {\it totally-real cubic cones}.

\vspace{1mm}

{\noindent All 8 cones defined by these three planes are said to be {\it cubic}.}
\end{definition}

\begin{remark}
From general theory of multidimensional continued fraction we have
following properties for cubic sails:
\begin{itemize}
\item A cubic sail is a polyhedral surface homeomorphic to $\r^2$.
Every face of that surface is bounded, every vertex of this sail is incident to
finitely many edges and faces.

\item The group $\Xi_+ (A)$ acts transitively on cubic sails for $A$,
sending vertices to vertices, edges to edges, and faces to faces.

\item There are finitely many orbits of vertices, edges and faces of every sail with respect to the action of $\Xi_+ (A)$.
\end{itemize}

Here we would like to note that in the higher dimensional case and
in the non-totally-real case similar constructions exist.
For further information we refer to~\cite{Arnold2002} and~\cite{KarpenkovGCF2013}.
\end{remark}

\section{Further lattice geometry used in the proofs}
\label{Lattice geometry used in the proofs}

In this section we prove several statement from lattice geometry
that are crucial for the proof. We also prove Theorem~\ref{eventually-periodic}$($i$)$
here.

\vspace{1mm}

In Subsection~\ref{Finiteness of planes splitting a given pyramid}
we establish finiteness of integer planes at fixed integer distance
splitting the cone into two parts one of which contains either a given pyramid
(Proposition~\ref{finiteness-planes})
or a given vertex of Klein's sail (Corollary~\ref{finiteness-for-faces}).
Further in Subsection~\ref{Pyramids of volumes bounded from below}
we prove uniform boundedness of volumes of sectional pyramids (Proposition~\ref{frozen-vertex})
and the areas of the corresponding sections (Corollary~\ref{S(T)-finite}) for algebraic cones.
In addition we prove here that the number of unit planes at unit integer distance to the origin that are cutting a pyramid of a separated from zero volume is finite
(Proposition~\ref{finiteness-unit-distance-planes-2}).
Further in Subsection~\ref{Lower bounds for areas and volumes}
we give a lover bounds on sectional area and volume via $\sin$ of the angle between
invariant planes (Proposition~\ref{area-bounded-alpha}).
In Subsection~\ref{Triangles with bounded coordinates}
we show finiteness for triangles in the plane with bounded coordinates (Proposition~\ref{bounded-triangle})
and discuss the estimates on coordinates for certain cases of states (Proposition~\ref{bounded-coords}).
Finally in Subsection~\ref{Theorem-i} we deduce the proof of Theorem~\ref{eventually-periodic}$($i$)$.

\subsection{Finiteness of planes splitting a given pyramid}
\label{Finiteness of planes splitting a given pyramid}

Despite the situation in the Euclidean geometry we have the following finiteness statement.

\begin{proposition}\label{finiteness-planes}
Let $P$ be an integer pyramid with vertex at the origin and with a triangular base.
Let $C$ be the integer cone defined by $P$ with vertex at the origin.
Finally, let $d$ be some positive real number.
Then there are only finitely many integer planes at integer distance at most $d$ from the origin $O$
dividing $C$ into two nonempty parts such that
one of these parts is bounded and contains $P$.
\end{proposition}

We start the proof with the following statement.

\begin{lemma}\label{finiteness-planes-lemma}
Let $C$ be the coordinate cone,
let $T$ be the unit coordinate tetrahedron,
and  let $d$ be some positive integer and $\varepsilon$ be some positive real numbers.
Then there are only finitely many integer planes at integer distance at most $d$ from the origin $O$
dividing $C$ into two nonempty parts such that
one of these parts is bounded and contains $\varepsilon$-dilate of the tetrahedron $T$ $($see Figure~\ref{p.2}$)$.
\end{lemma}

\begin{figure}
\centering
\includegraphics{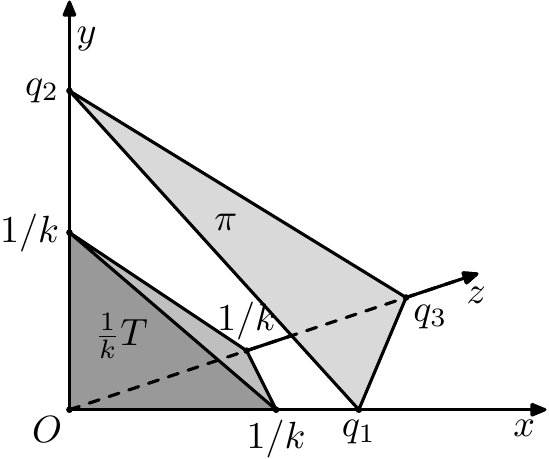}
\caption{An admissible intersection of a plane $\pi$ with the cone $C$.}
\label{p.2}
\end{figure}

\begin{proof}
Let some plane $\pi$ satisfy all the conditions of the lemma.
Assume that the coordinates of intersection of $\pi$ with axes are
$$
q_1=\frac{m_1}{n_1}, \qquad
q_2=\frac{m_2}{n_2}, \quad
\hbox{and}
\quad
q_3=\frac{m_3}{n_3}.
$$

\vspace{1mm}

First of all let us estimate $m_1$ and $n_1$.
Consider two cases for $q_1=m_1/n_1$.

\vspace{1mm}

{
\noindent
{\it Case 1.}}
Let $q_1>1$.
Then the point $(1,0,0)$ divides the segment with endpoints
$(0,0,0)$, $(q_1,0,0)$ in proportion $n_1:m_1$.
Assuming that $n_1$ and $m_1$ are relatively prime, we have at least
$$
n_1+m_1-1
$$
integer planes parallel to $\pi$ that are between $\pi$ and the origin.
Hence $n_1,m_1\le d$.

\vspace{1mm}

{
\noindent
{\it Case 2.}}
Let now $\varepsilon <q_1\le 1$.
Then the point $(q_1,0,0)$ divides the segment with endpoints
$(0,0,0)$, $(1,0,0)$ in proportion $m_1:(n_1-m_1)$.
Assuming that $n_1$ and $m_1$ are relatively prime, we have at least
$
m_1
$
integer planes parallel to $\pi$ that are between $\pi$ and the origin.
Hence $m_1\le d$.
Now we use the fact that
$$
\frac{m_1}{n_1}>\varepsilon
$$
to get
$$
n_1<\frac{m_1}{\varepsilon}<\frac{d}{\varepsilon}.
$$

\vspace{1mm}

For both cases we have
$$
m_1<d \qquad \hbox{and} \qquad n_1<\frac{d}{\varepsilon}.
$$

Similarly, we have $n_i\le d$ and $m_i\le d/\varepsilon$ for $i=2,3$.

\vspace{2mm}

Therefore, we have at most $d^6/\varepsilon^3$ integer planes
satisfying the conditions of the lemma.
\end{proof}

Let us deduce the proof of Proposition~\ref{finiteness-planes}.

\vspace{2mm}

{
\noindent
{\it Proof of Proposition~\ref{finiteness-planes}.}
First of all note that the statement of Lemma~\ref{finiteness-planes-lemma} is invariant under integer affine transformations.
Therefore, it holds for any integer pyramid $T$ whose edges generate the lattice $\z^3$.

\vspace{2mm}

Secondly, there exists an integer $k$ such that the $k$-tuple pyramid $kP$ contains a triple of vectors that generate the basis. Denote the resulting pyramid by $T'$.
Now, note that  $1/k\cdot T'$ is contained in $P$.

It is clear that any integer plane satisfying conditions of this corollary also satisfy them for $1/k\cdot T'$.
By Lemma~\ref{finiteness-planes-lemma} the number of such planes is finite.
Therefore, the number  of planes is finite for $P$ as well.

\qed
}

Let us fix the following notation.

\begin{definition}
Let a plane $\pi$ divide a cone $C$ centered at the origin
onto two non-empty parts one of which is finite.
Then we say that $\pi$ is a {\it base} plane for $C$.
We denote the finite part by $\pyr(C,\pi)$ (see Figure~\ref{p.1}).
\end{definition}

\begin{figure}
 \centering
 \includegraphics{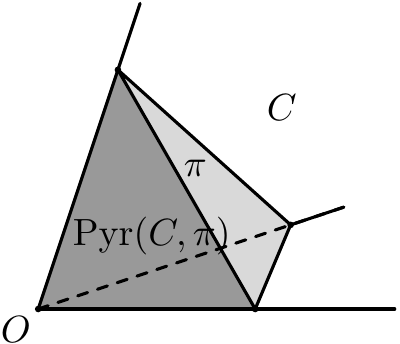}
\caption{A cone $C$ together with one of its base planes $\pi$ and the corresponding pyramid $\pyr(C,\pi)$.}
\label{p.1}
\end{figure}

\begin{corollary}\label{finiteness-for-faces}
Consider a  cone $C$ in $\r^3$ centered at the origin.
Let $v$ be any point in the interior of $C$, and let $d$ be a positive real number.
Then there are only finitely many integer base planes $\pi$ satisfying the following two conditions:
$$
\id(O,\pi)<d \qquad \hbox{and} \qquad v\subset \pyr(C,\pi).
$$
\end{corollary}

\begin{proof}
Let $v_i$ for $i=1,2,3$ be the projections of $v$ to edges $e_{i}$ of $C$
along the corresponding 2-faces of $C$ not containing $e_i$ for $i=1,2,3$.
Denote
$$
P_v=Ov_1v_2v_2.
$$
The pyramid $P_v$ is included into every pyramid $\pyr(C,\pi)$
satisfying the conditions of the corollary, since $\pi$ is a base plane and
since $v\in \pyr(C,\pi)$.

Now the finiteness of such planes follows directly from Proposition~\ref{finiteness-planes}.
\end{proof}

\subsection{Finiteness and boundedness results for algebraic cones}
\label{Pyramids of volumes bounded from below}

We start with one topological definition.

\begin{definition}
Let $P$ be a polyhedral surface and let $v$ be one of its vertices.
Denote by $\starr(v)$ the union of all edges emanating from $v$.
We call it by a (one-dimensional) {\it star} of $P$ at $v$.
\end{definition}

Now we are ready to formulate the statement on triangles inside $\pyr(C,\pi)$
for $\pi$ at sufficiently large integer distances from the origin.

\begin{proposition}\label{frozen-vertex}
Let $C$ be a totally-real cubic cone in $\r^3$ centered at the origin.
Then  there exist a positive real number $M_1$ such that the following statements hold.

\vspace{1mm}

\noindent{\it $($i$)$}
For every
integer base plane $\pi$ satisfying
$$
\vol(\pyr(C,\pi))>M_1
$$
the pyramid $\pyr(C,\pi)$ contains an integer vertex of the Klein's sail $K(C)$ it its interior.

\vspace{1mm}

\noindent{\it $($ii$)$}
For every integer base plane $\pi$ at the unit integer distance to the origin
we have
$$
\vol\big(\pyr(C,\pi)\big)<M_1.
$$

\end{proposition}

We start the proof with the following lemma.

\begin{lemma}\label{finiteness-algebraic-lemma}
Let $C$ be a totally-real cubic cone in $\r^3$ centered at the origin.
Then there exist positive constants $M_2$, $M_3$ and $M_4$
$($depending entirely on $C$$)$ such that the following uniform
properties hold:
\begin{itemize}
\item
For every vertex $v$ of $K(C)$ the number of adjacent to $v$ faces of $K(C)$ $($and, therefore,
the number of the adjacent edges as well$)$
are bounded by some number $M_2$;

\item
For every plane $\pi$ containing a face of $K(C)$ we have
$$
M_3<\vol(\pyr(C,\pi))<M_4;
$$
\end{itemize}
\end{lemma}

\begin{remark}
In fact a real number $M_1$ in Proposition~\ref{frozen-vertex}
can be taken as
$$
M_1=M_2 M_4,
$$
where $M_2$ and $M_4$ are as in Lemma~\ref{finiteness-algebraic-lemma}.
\end{remark}

{
\noindent
{\it Proof of Lemma~\ref{finiteness-algebraic-lemma}.}
The first item is a classical statement, see, e.g., Theorem~16.29 in~\cite{KarpenkovGCF2013}.
}

\vspace{2mm}

The second statement follows from
the invariance of volumes under the action of the positive Dirichlet group $\Xi_+(C)$
together with the fact
that the number of faces and vertices is finite
up to the action of $\Xi_+(C)$ (see e.g. Chapter~18 of~\cite{KarpenkovGCF2013}).
\qed

\vspace{2mm}

{
\noindent
{\it Proof of Proposition~\ref{frozen-vertex}.}
{\it Item~$($i$)$.}
Let $v\in K(C)$ be one of the closest integer point to the origin with respect to $\pi$.
Namely,
$$
\big(\pyr(C,\pi') \setminus \pi'\big)\cap \z^3=\{O\},
$$
where $\pi'$ is the plane passing through $v$ and parallel to $\pi$.
}

Now denote by $\pi_i$ all the planes of faces adjacent to $v$.
By Lemma~\ref{finiteness-algebraic-lemma} there are at most $M_1$ of such faces (note that $M_1$
depends only on the cone but not on the vertex).
Set
$$
P_i=\pyr(C,\pi_i).
$$
and let
$$
E=\bigcup\limits_{i} P_i.
$$

From Lemma~\ref{finiteness-algebraic-lemma} we know that $\vol(P_i)<M_4$ for all $P_i$
in the union.
Therefore,
$$
\vol(E)<M_2 M_4
$$
and hence it is uniformly bounded for all vertices.

\vspace{1mm}

By construction $\pyr(C,\pi')\subset E$, and hence,
$$
\vol(\pyr(C,\pi'))<M_2 M_4.
$$
Therefore, the statement of Item~$($i$)$ holds for $M_1=M_2 M_4$.

\vspace{2mm}

{\noindent{\it Item~$($ii$)$.}} Let $\pi$ satisfy
$$
\vol\big(\pyr(C,\pi)\big)<M_1.
$$
Then by Item~$($i$)$ there exists an integer point in the interior of the pyramid $\pyr(C,\pi)$.
Therefore, the integer distance from the origin to $\pi$ is at least 2.
This concludes the proof of Item~$($ii$)$.
\qed

Let us fix the following important notation.

\begin{definition}
Consider an integer plane $\pi$ and a convex polygon $P$ in it.
Let $S$ be the basis square of the integer lattice in it.
We say that the quantity
$$
\frac{\area (P)}{\area (S)}
$$
is a $\pi$-area of $P$ and denote it by $\area_\pi(P)$.
\end{definition}

Let us now deduce uniform boundedness of $\pi$-area for bounded sections
on unit integer distance.

\begin{corollary}\label{S(T)-finite}
Consider  a totally-real cubic cone in $\r^3$ centered at the origin.
Then there exists a real constant $M$ such that
for every integer base plane $\pi$
  at the unit  integer distance to the origin
we have
$$
\area_\pi(C\cap \pi)<M.
$$
$($Recall that for any base plane $\pi$, the triangle  $C\cap \pi$ is bounded.$)$
\end{corollary}

\begin{proof}
The corollary follows directly from the fact that
the volume of a tetrahedron is equal to a third of its integer distance to the base times $\pi$-volume
of the base. We can actually take $M=M_1/3$, where $M_1$ is as in Proposition~\ref{frozen-vertex}.
\end{proof}

Finally we prove the following result.

\begin{proposition}\label{finiteness-unit-distance-planes-2}
Consider  a totally-real cubic cone in $\r^3$ centered at the origin.
Let $\varepsilon$ be a positive real number.
Then the number of integer base planes $\pi$ on the unit integer distance
to the origin satisfying
$$
\vol\big(\pyr(C,\pi)\big)>\varepsilon
$$
is finite up to the action of the positive Dirichlet group $\Xi_+(C)$.
\end{proposition}

\begin{proof}
Set
$$
k=\lfloor M_1/\varepsilon\rfloor+1,
$$
where $M_1$ is the constant of Proposition~\ref{frozen-vertex}$($i$)$ for the cone $C$.

By Proposition~\ref{frozen-vertex}$($i$)$
for every integer base plane $\pi$ with $\vol\big(\pyr(C,\pi)\big)>\varepsilon$,
the pyramid $k\pyr(C,\pi)$ (i.e. the $k$-dilate of $\pyr(C,\pi)$) contains some vertex $v$ of the sail $K(C)$.

Note that there are only finitely many vertices of the sail up to the action of the
positive Dirichlet group $\Xi_+(C)$ (for more information see~\cite{Arnold2002}
and~\cite{KarpenkovGCF2013}).
By Corollary~\ref{finiteness-for-faces} for each of the choices of the sail vertex there are only finitely many integer planes $\pi'$ on integer distance $k$  containing this vertex in $\pyr(C,\pi').$
Therefore, there are only finitely many
integer planes $\pi$ on the unit integer distance
to the origin satisfying
$$
\vol\big(\pyr(C,\pi)\big)>\varepsilon
$$
up to the action of the positive Dirichlet group $\Xi_+(C)$.
\end{proof}

\subsection{Lower bounds for areas and volumes}
\label{Lower bounds for areas and volumes}

In this section we relate the angle between invariant planes from
the one hand and the areas and volumes of corresponding sectional pyramids from the other hand.

\begin{proposition}\label{area-bounded-alpha}
Consider a state
$$
s=(\xi=(x_0,y_0,1),\nu_1=(x_1,y_1,1),\nu_2=(x_2,y_2,1))
$$
in a separating basis. Denote by $T$ the triangle with vertices at endpoints of
$\xi$ , $\nu_1$, and $\nu_2$.
Let $\alpha$ be the angle between the invariant planes containing $\xi$,
then the area of the triangle $T$
and the volume of the pyramid $\pyr(T)$ with base at $T$ and vertex at the origin are estimated as follows
$$
\area(T)> \frac{1}{4}\sin \alpha, \quad \hbox{and} \quad
\vol(\pyr(T))> \frac{1}{12}\sin \alpha.
$$
\end{proposition}

We start the proof with the following three lemmas.

\begin{lemma}\label{area-bounded-alpha-lemma-1}
Let $OABC$ be a tetrahedron with coordinates
$$
O=(0,0,0), \quad A=(a,0,1), \quad B=(b,0,1), \quad \hbox{and} \quad C=(x_0,y_0,1),
$$
where $b>a$ and $y_0>1$.
Denote the angle between the planes $OAC$ and $OBC$ by $\alpha$.
Then
$$
\area(ABC)> \frac{1}{4}\sin \alpha.
$$
\end{lemma}

\begin{proof}
Note that the angle between planes is equal to the angle between its normals.
In particular we have the following formula (here by $v\times w$
we denote the cross product of two vectors $v$ and $w$):
$$
\sin \alpha=\frac{|(OA\times OC)\times(OB\times OC)|}{|OA\times OC|\cdot|OB\times OC|}.
$$

Note that
$$
\begin{array}{l}
OA\times OC=(-y_0,x_0-a,ay_0);\\
OB\times OC=(-y_0,x_0-b,by_0),
\end{array}
$$
and hence
$$
(OA\times OC)\times(OB\times OC)=(b-a)y_0(x_0,y_0,1).
$$
Let us estimate the norms of the above three vectors.
We have:

\begin{align*}
(|OA\times OC|)^2&=y_0^2(a^2+1)+(x_0-a)^2\\
&\ge y_0^2+a^2+(x_0-a)^2\\
&=y_0^2+\frac{x_0^2}{2}+\Big(\frac{x_0}{\sqrt{2}}-\sqrt{2}a \Big)^2\\
&> \frac{x_0^2+y_0^2+1}{2}.
\end{align*}

Similarly (by swapping $a$ with $b$) we have
$$
(|OB\times OC|)^2>\frac{x_0^2+y_0^2+1}{2}.
$$

Direct computations give us the following expression
$$
(|(OA\times OC)\times(OB\times OC)|)^2=(b-a)^2y_0^2(x_0^2+y_0^2+1).
$$
Substituting to the formula for the $\sin \alpha$ above we have
$$
\sin\alpha=\frac{(b-a)y_0(x_0^2+y_0^2+1)^{1/2}}{|OA\times OC|\cdot|OB\times OC|}.
$$
Therefore,
$$
\area(ABC)=\frac{1}{2}
y_0(b-a)=\frac{1}{2}\sin\alpha\frac{|OA\times OC|\cdot|OB\times OC|}{(x_0^2+y_0^2+1)^{1/2}}.
$$
Finally applying the estimates for $|OA\times OC|$ and $|OB\times OC|$
we have
$$
\area(ABC)>\frac{1}{2}\sin\alpha \frac{\big(\frac{x_0^2+y_0^2+1}{2}\big)^{1/2}{\big(\frac{x_0^2+y_0^2+1}{2}\big)^{1/2}}}{(x_0^2+y_0^2+1)^{1/2}}
=\frac{1}{4}\sin\alpha(x_0^2+y_0^2+1)^{1/2}>\frac{1}{4}\sin\alpha.
$$
This concludes the proof.
\end{proof}

\begin{lemma}\label{area-bounded-alpha-lemma-2}
Let $OABC$ be a tetrahedron with coordinates
$$
O=(0,0,0), \quad A=(0,a,1), \quad B=(0,b,1), \quad \hbox{and} \quad C=(x_0,y_0,1),
$$
where $b>a$ and $y_0>1$.
Denote the angle between the planes $OAC$ and $OBC$ by $\alpha$.
Then
$$
\area(ABC)> \frac{1}{4}\sin \alpha.
$$
\end{lemma}

\begin{proof}
The proof repeats the proof of Lemma~\ref{area-bounded-alpha-lemma-1} after swapping the first two coordinates.
\end{proof}

\begin{lemma}\label{area-bounded-alpha-lemma-3}
Let $OABC$ be a tetrahedron with coordinates
$$
O=(0,0,0), \quad A=(0,a,1), \quad B=(b,0,1), \quad \hbox{and} \quad C=(x_0,y_0,1),
$$
where $a\ge 0$, $b\ge 0$, $x_0>1$, and $y_0>1$.
Denote the angle between the planes $OAC$ and $OBC$ by $\alpha$.
Then
$$
\area(ABC)> \frac{1}{4}\sin \alpha.
$$
\end{lemma}

\begin{proof}
Direct calculations shows that
$$
\begin{array}{rcl}
\area^2(ABC)&=&\displaystyle \frac{(ax_0+by_0-ab)^2}{4};\\
\sin^2\alpha&=&\displaystyle{\frac { \left( x_0^2+y_0^2+1 \right)  \left(ax_0+by_0-ab \right) ^{2}}{
 \left( (x_0-b)^2+(1+b^2)y_0^2 \right)  \left( (1+a^2)x_0^2+(y_0-a)^2 \right) }}.
\end{array}
$$
Therefore,
$$
\begin{aligned}
\area^2(ABC)&=\frac{ \left( (x_0-b)^2+(1+b^2)y_0^2 \right)  \left( (1+a^2)x_0^2+(y_0-a)^2 \right)}
{4\left( x_0^2+y_0^2+1 \right)}
\sin^2\alpha\\
&>\frac{1}{4} \cdot\frac{  x_0^2y_0^2 }
{x_0^2+y_0^2+1}
\sin^2\alpha
=
\frac{1}{4} \cdot\frac{  1 }
{\frac{1}{x_0^2}+\frac{1}{y_0^2}+\frac{1}{x_0^2y_0^2}}
\\
&>\frac{1}{12}\sin^2\alpha>\frac{1}{4^2}\sin^2\alpha,
\end{aligned}
$$
the second inequality holds since $x_0>1$ and $y_0>1$.
Therefore,
$$
\area(ABC)>\frac{1}{4}\sin \alpha,
$$
this concludes the proof.
\end{proof}

{\noindent
{\it Proof of Proposition~\ref{area-bounded-alpha}.}}
Since the basis is separating we have:
\\
{\it $($i$)$} The first two coordinates of $\xi$ are greater than 1;
\\
{\it $($ii$)$} The vectors $\nu_1$ and $\nu_2$  both have at least one negative coordinate.

\vspace{1mm}

Denote by $H$ the union of two closed rays
$$
H=\{(t,0,1)|t\ge 0\}\cup\{(0,t,1)|t\ge 0\}.
$$
Let $A$  and $B$ be the intersection points of the lines $\xi\nu_1$
and $\xi\nu_2$ with $H$ respectively.
Since the basis is separating for the state $s$, the point $A$ is inside the segment $\xi\nu_1$
and the point $B$ is inside the segment $\xi\nu_2$.

\vspace{1mm}

Now we can apply Lemmas~\ref{area-bounded-alpha-lemma-1},
~\ref{area-bounded-alpha-lemma-2}, and~~\ref{area-bounded-alpha-lemma-3},
we have
$$
\area(\xi\nu_1\nu_2)>\area(\xi AB)>\frac{1}{4}\sin \alpha.
$$
Finally since the distance from the origin to the base $T$ is 1, we have
$$
\vol(\pyr((0,0,0)\xi\nu_1\nu_2))> \frac{1}{12}\sin \alpha.
$$
This concludes the proof of Proposition~\ref{area-bounded-alpha}.
\qed

\subsection{Triangles with bounded coordinates}
\label{Triangles with bounded coordinates}

Let us start with the following proposition.

\begin{proposition}\label{bounded-triangle}
Let $T$ be an arbitrary triangle on an integer plane $\pi$,
and $M>0$ be a real number.
Then there exists finitely many integer affine bases $(O,e_1,e_2)$ in which
all absolute values of the coordinates of vertices of $T$ are bounded by $M$ from above.
\end{proposition}

\begin{proof}
Fix some integer basis $(O,e_1,e_2)$ of the plane $\pi$.
Consider another integer basis $(O',e_1',e_2')$ of $\pi$ satisfying the condition of the proposition.
Let $S$ be a transition matrix from $e_1, e_2$ to $e_1', e_2'$.

\vspace{1mm}

Let us show that the elements of the matrix $S$ are all bounded.
There exists a real number $\varepsilon>0$ such that the triangle $T$ contains a triangle $T_\varepsilon$
with vertices
$$
v_1, \quad v_2=v_1+(\varepsilon,0), \quad v_3=v_1+(0,\varepsilon).
$$
Let $B$ be the matrix of vectors $v_2-v_1$ and $v_3-v_1$.
It is clear that
$$
B=\varepsilon
\left(
\begin{array}{cc}
1&0\\
0&1\\
\end{array}
\right).
$$
Then after the change of coordinates the new coordinates of the differences of vectors will be
$$
S^{-1}B=\varepsilon S^{-1}.
$$
In case if some element of $S^{-1}$ are greater than
$\frac{2M}{\varepsilon}$,
then at least one of the coordinates of $T_\varepsilon$ is greater than $M$.
Hence at least one of the coordinates of $T$ is greater than $M$.

So the absolute values of the elements of $S^{-1}$
do not exceed $\frac{2M}{\varepsilon}$.
Since $\det S=1$, the elements of $S$ do not exceed $\frac{2M}{\varepsilon}$ as well.
There are only finitely many integer matrices with bounded coefficients from above.

\vspace{1mm}

Therefore, only finitely many bases satisfy the condition of the proposition up to shifts on integer vector.

\vspace{2mm}

Let now $(O',e_1',e_2')$ and $(O'',e_1',e_2')$ satisfy the condition of the proposition.
Then the coordinates  of $O''$ in the basis $(O',e_1',e_2')$ are bounded by $2M$.
Thus for a fixed $e_1'$ and $e_2'$
there are only finitely many choices of $O'$ such that $(O',e_1',e_2')$ satisfies the condition of the proposition.

\vspace{2mm}

Therefore, only finitely many bases satisfy the condition of the proposition.
\end{proof}

\begin{proposition}\label{bounded-coords}
Consider a separating state
$$
s=\big(\xi=(x_0,y_0,1),\nu_1=(x_1,y_1,1),\nu_2=(x_2,y_2,1)\big).
$$
Let
$$
\sin\alpha(s)>\varepsilon \quad \hbox{and} \quad
\vol\big((0,0,0),\xi,\nu_1,\nu_2\big)<M.
$$
Then
$$
|x_0|, |y_0|, |x_1|, |y_1|, |x_2|,|y_2|<\frac{6\sqrt{2}M}{\varepsilon}.
$$
\end{proposition}

\begin{proof}
Let us first estimate the distance from the point $\nu_2$ to the line $\ell$
spanned by $(0,0,0)$ and $\xi$.
Note that $\nu_2$ has either $x_2<0$ or $y_2<0$ as the state $s$ is separating.
Hence the minimum of the distance does not exceed the minimum between the line $\ell$ and the union of lines
$(t,0,1)$ and $(0,t,1)$ with parameter $t$.
Projecting along these lines (and keeping in mind that $x_0>y_0>1$) we get the estimate
$$
\dist(\nu_2,\ell)\ge \frac{\sqrt{2}}{2}.
$$
(Here and below we denote by $\dist(S_1,S_2)$ the Euclidean distance between $S_1$ and $S_2$.)
Denote by $\pi_1$ the plane spanned by $(0,0,0)$, $\xi$, and $\nu_1$.
We have
$$
\dist(\nu_2,\pi_1)=\dist(\nu_2,\ell)\sin\alpha(s)> \frac{\sqrt{2}\varepsilon}{2}.
$$
Let us estimate now the area of the triangle $T_1$ spanned by $(0,0,0)$, $\xi$, and $\nu_1$:
$$
\area(T_1)=3\frac{\vol\big((0,0,0),\xi,\nu_1,\nu_2\big)}{\dist(\nu_2,\pi_1)}<\frac{3\sqrt{2}M}{\varepsilon}.
$$
From the vector product formula we have:
$$
\area(T_1)=|\xi\times\nu_1|=|(x_0y_1-y_0x_1,y_0-y_1,x_1-x_0)|.
$$
Denote $C=3\sqrt{2}M/\varepsilon.$
It is clear that
$$
|x_0y_1-y_0x_1|<C, \qquad |x_0-x_1|<C, \qquad |y_0-y_1|<C.
$$

\vspace{1mm}

Assume that $x_1$ and $y_1$ are negative. Then we immediately have
$$
|x_0|, |y_0|, |x_1|, |y_1|<C
$$
as $x_0$ and $y_0$ are positive.

\vspace{1mm}

Assume now that one of $x_1$ and $y_1$ is non-negative. Therefore,
$$
|x_1|+|y_1|\le |x_0y_1|+|y_0x_1|=|x_0y_1-y_0x_1|<C.
$$
(the first inequality holds as $x_0>y_0>1$).
Hence  $x_1<C$ and $y_1<C$, and, therefore, $x_0<2C$ and $y_0<2C$.

\vspace{1mm}

The case of positive $x_1$ and $y_1$ is empty for separating states.

\vspace{1mm}

Therefore,
$$
|x_0|, |y_0|, |x_1|, |y_1|<\frac{6\sqrt{2}M}{\varepsilon}.
$$

\vspace{2mm}

The proof of the fact that $|x_2|, |y_2|<6\sqrt{2}M/\varepsilon$ repeats the above proof
of $|x_1|, |y_1|<6\sqrt{2}M/\varepsilon$.
It is omitted here.
\end{proof}

\subsection{Proof of Theorem~\ref{eventually-periodic}$($i$)$: finiteness of  $\Omega_\varepsilon(\xi;\nu_1,\nu_2)$}\label{Theorem-i}

Let us prove finiteness of the set $\Omega_\varepsilon(\xi;\nu_1,\nu_2)$ for an arbitrary positive $\varepsilon$.

\vspace{2mm}

\noindent{
{\it Proof of Theorem~\ref{eventually-periodic}$($i$)$.}
Without loss of generality we assume that we are given by a separating
$\xi$-state $\hat s=(\hat\xi,\hat\nu_1, \hat \nu_2)$
(one should swap coordinate vectors and normalize vectors $(\xi,\nu_1,\nu_2)$ such that the last coordinates
of the vectors are all units).

Let $\alpha(\hat s)>\varepsilon$, then by Proposition~\ref{area-bounded-alpha}
we have
$$
\vol\big((0,0,0),\hat \xi,\hat \nu_1,\hat \nu_2\big)>\frac{1}{12}\varepsilon.
$$
Hence by Proposition~\ref{finiteness-unit-distance-planes-2}
the number of integer base planes $\pi$ on the unit integer distance
to the origin satisfying
$$
\vol\big((0,0,0),\hat\xi,\hat \nu_1,\hat\nu_2\big)>\frac{1}{12}\varepsilon
$$
is finite up to the action of the positive Dirichlet group $\Xi_+(C)$.
}

Let us now fix one of such planes $\pi$.
(Here we would like to note that the planes of the same orbit of $\Xi_+(C)$
have the same set of separating states. So it is sufficient to consider only one plane for each
of the orbits.)
Let
$$
\tilde s=\big(\xi=(\tilde x_0,\tilde y_0,1),\nu_1=(\tilde x_1,\tilde y_1,1),\nu_2=(\tilde x_2,\tilde y_2,1)\big).
$$
be one of the separating $\xi$-states in an integer basis where $\pi$ is given by $z=1$.
By Proposition~\ref{frozen-vertex}$($ii$)$
there exists a uniform real constant $M_1$ such that
$$
\vol\big((0,0,0),\tilde\xi,\tilde\nu_1,\tilde\nu_2\big)<M_1.
$$
Hence by Proposition~\ref{bounded-coords} we have
$$
|\tilde x_0|, |\tilde y_0|, |\tilde x_1|, |\tilde y_1|<\frac{6\sqrt{2}M_1}{\varepsilon}.
$$
Therefore, by Proposition~\ref{bounded-triangle}
there are finitely many integer affine bases $(O,e_1,e_2)$ for the plane $\pi$ in which
all absolute values of the coordinates of vertices of $\tilde \xi$, $\tilde \nu_1$, and $\tilde \nu_2$ are bounded from above.

\vspace{2mm}

Hence, there are finitely many choices of $(O,e_1,e_2)$ such that $\alpha(s)>\varepsilon$.
Therefore, for every $\varepsilon>0$ the set $\Omega_\varepsilon(\xi;\nu_1,\nu_2)$ is finite.
\qed

\section{Formulation of  Theorem~\ref{greater e}, its equivalence to Theorem~\ref{eventually-periodic}$($ii$)$}\label{Theorem-ii}

In this section we formulate
Theorem~\ref{greater e} and show its equivalence to Theorem~\ref{eventually-periodic}$($ii$)$.
We prove~\ref{greater e}  later  in Sections~\ref{teor-v}--\ref{theor-iii-iv}.

\vspace{2mm}

%

\myexample{Preliminary set-up}
{
Let $A\in GL(3,\z)$ be a matrix with irreducible over $\q$ characteristic polynomial
with three real eigenvalues. Let $\xi$ be one of its eigenvectors,
and let $\nu_1,\nu_2$ be the other two eigenvectors.
Let us prove that $\Omega_{\max}(\xi;\nu_1,\nu_2)$ is finite.
In another words,
there are finitely many separating $\xi$-states satisfying
$$
\sin^2\alpha(s)>\sin^2\alpha(\Phi(s)).
$$

\vspace{2mm}

Consiuder a separating $\xi$-state $s$ and  let
$$
s=(\tilde \xi(x_0,y_0,1),\tilde\nu_1(x_1,y_1,1),\tilde\nu_2(x_2,y_2,1)).
$$

Let $T(s)$ denotes the triangle with vertices $\tilde\xi$, $\tilde\nu_1$, and $\tilde\nu_2$ (note that they are all
in the plane $z=1$).

Let the line $\tilde\xi \tilde\nu_i$ intersects the coordinate axes at points denoted by
$$
P_i(0,p_i,1), \quad \hbox{and} \quad Q_i(q_i,0,1)
$$
for $i=1,2$.

\vspace{2mm}

Since the state $s$ is separating,
the edges $\tilde\xi \tilde\nu_1$ and $\tilde\xi\tilde\nu_2$ intersect the boundary of the positive quadrant of the plane $z=1$.
Let us denote these points of intersections by $P$ and $Q$.
Then for the points $P$ and $Q$ (up to their swap) we have the following three cases.

\vspace{1mm}

\begin{itemize}
\item Case I: $M=(0,m,1)$ and $N=(0,n,1)$ where $m>n\ge 0$.

\item Case II: $M=(m,0,1)$ and $N=(0,n,1)$ where $m,n>0$.

\item Case III: $M=(m,0,1)$ and $N=(n,0,1)$ where $n>m>0$.
\end{itemize}

}

\vspace{2mm}

Let us now reformulate Theorem~\ref{eventually-periodic}$($ii$)$ in terms of the above settings.

\begin{theorem}\label{greater e}
The following five statements hold:

\vspace{1mm}
{\noindent $($i$)$}
For every $\varepsilon>0$ there are finitely many elements in $\Omega_{\max}$ satisfying
$$
\min(|p_2-p_1|,|q_2-q_1|)\ge\varepsilon.
$$

{\noindent $($ii$)$}
Let now $\varepsilon\le 0.05$  and
$$
\min(|p_2-p_1|,|q_2-q_1|)<\varepsilon.
$$
Then we have:

--- either $|p_2-p_1|<20\varepsilon$ and $p_1,p_2>-1.1$,

--- or $|q_2-q_1|<\varepsilon$ and $q_1,q_2>-0.1$.

--- or  $p_1,p_2\le -1.05$  and for $i$ with minimal $q_i$ it holds:
$$
q_i\ge \frac{q_i}{|p_i|}+1.
$$

{\noindent $($iii$)$}
There exists $\varepsilon_1>0$ $($that depends entirely on $\xi$$)$ such that if
$$
|p_2-p_1|<\varepsilon_1 \quad \hbox{and} \quad p_1,p_2>-1.1
$$
then the corresponding state $s$ is not in $\Omega_{\max}$.

\vspace{1mm}
{\noindent $($iv$)$}
There exists $\varepsilon_2>0$ $($that depends entirely on $\xi$$)$ such that if
$$
|q_2-q_1|<\varepsilon_2 \quad \hbox{and} \quad q_1,q_2>-0.1
$$
then the corresponding state $s$ is not in $\Omega_{\max}$.

\vspace{1mm}
{\noindent $($v$)$}
If $p_1,p_2\le -1.05$  and for $i$ with minimal $q_i$ it holds:
$$
q_i\ge \frac{q_i}{|p_i|}+1.
$$
Then the corresponding state $s$ is not in $\Omega_{\max}$.
\end{theorem}

\begin{remark}\label{equivalence-theorems}
Indeed from the last four items there exist $\varepsilon_1$ and $\varepsilon_2$ such that for
$$
\varepsilon_0=\min(0.05, \varepsilon_1/20,\varepsilon_2)
$$
it holds: if the state $s$ satisfies
$$
\min(|p_2-p_1|,|q_2-q_1|)< \varepsilon_0
$$
then it is not in  $\Omega_{\max}$.
On the other hand the number of states in  $\Omega_{\max}$ satisfying
$$
\min(|p_2-p_1|,|q_2-q_1|)\ge \varepsilon_0
$$
is finite. Therefore Theorem~\ref{eventually-periodic}$($ii$)$ is equivalent to Theorem~\ref{greater e}.
\end{remark}

{\noindent {\bf
The proof of Theorem~\ref{greater e} is organized as follows.}
We start with Item~$($v$)$ in Section~\ref{teor-v}.
The  proof of Item~$($i$)$ is splitted into these three cases
(See Cases~I, II, and~III as above),
it is given  in Section~\ref{teor-i}.
Item~$($ii$)$  is proven in Section~\ref{teor-ii}.
Finally in Section~\ref{theor-iii-iv} we show Items~$($iii$)$ and~$($iv$)$.
Some estimates in the proof were computed in MAPLE2020 (see in~\cite{maple}).
}


\section{Proof of Theorem~\ref{greater e}~$($v$)$}\label{teor-v}

\begin{proposition}\label{fin-3-5}
Let $a\ge k{+}1> 1$, $t\ge 1$.
Then for any $\rho>0$  the operator $V_{1,0;0}$  increases $\sin^2\alpha$
for the $\xi$-state with
$$
\xi=(a+tk,t,1), \qquad Q_1=(a,0,1), \quad   \hbox{and} \quad Q_2=(a+\rho,0,1).
$$
\end{proposition}

\begin{proof}
First of all note that $V_{1,0;0}$ sends the whole set of admissible states to the positive octant.

\vspace{2mm}

Denote
$$
Q(a,k,t,\varepsilon,\rho)=\frac{\sin^2\alpha(V_{\varepsilon,0;0}(s))}{\sin^2\alpha(s)}.
$$

Let us prove that for $a\ge k>0$ and $t\ge 1$ the derivative
$$
\frac{\partial Q(a,k,t,\varepsilon,\rho)}{\partial \varepsilon}>0.
$$

Note that
$$
\begin{array}{l}
\frac{\partial Q(a,k,t,\varepsilon,\rho)}{\partial \varepsilon}=
\frac{2(a^2t + a\rho t + k^2t - k\rho + t)p(a,k,t,\varepsilon,\rho)}
{(k^2 t^2 + 2 a k t + a^2 + t^2 + 1) (a^2 + k^2 + 1) (a^2 t^2 + 2 a \rho t^2 + k^2 t^2 + \rho^2 t^2 - 2 k \rho t + \rho^2 + t^2)}
\end{array},
$$
where
$$
\begin{aligned}
p(a,k,t,\varepsilon,\rho)=&
2 a k^2 t^3 + k^2 \rho t^3 + 3 a^2 k t^2 + a k \rho t^2 - k^3 t^2 +a^3 t - a k^2 t + 2 a t^3
\\&
 + k^2 \rho t + \rho t^3 + a k \rho - k t^2 + a t + \rho t.
\end{aligned}
$$

It is clear that the factors in the denominator of the expression for $\frac{\partial Q}{\partial \varepsilon}$ are positive
as they are the sums of squares. Further, since $t\ge 1$, $a\ge k>0$,
and $\rho>0$ we have
$$
a^2t + a\rho t + k^2t - k\rho + t>a\rho t  - k\rho \ge a\rho   - k\rho \ge 0.
$$
Finally, let us substitute $a=k+\nu$ (here $\nu\ge 0$) to $p(a,k,t,\varepsilon,\rho)$.
We have
$$
\begin{aligned}
p(k+\nu,k,t,\varepsilon,\rho)=&
2 k^3 t^3 + 2 k^2 \nu t^3 + k^2 \rho t^3 + 2 k^3 t^2 + 6 k^2 \nu t^2 + k^2 \rho t^2 + 3 k \nu^2 t^2 + k \nu \rho t^2 +
 \\&
 2 k^2 \nu t + k^2 \rho t + 3 k \nu^2 t + 2 k t^3 + \nu^3 t + 2 \nu t^3 + \rho t^3 + k^2 \rho + k \nu \rho - k t^2 +
 \\& k t + \nu t + \rho t.
\end{aligned}
$$
There is only one negative term in the above expression: $- k t^2$, whose absolute value is always smaller than
for the term  $2 k t^3$. Hence
$$
p(a,k,t,\varepsilon,\rho)>0
$$
for $t\ge 1$, $a\ge k>0$ and $\rho>0$.

Thus for $t\ge 1$, $a\ge k>0$ and $\rho>0$ we have
$$
\frac{\partial Q(a,k,t,\varepsilon,\rho)}{\partial \varepsilon}>0.
$$

Integrating the last expression we have: for $t\ge 1$, $a> k+1\ge 1$ and $\rho>0$ it holds
$$
Q(a,k,t,1,\rho)>0,
$$
which implies the fact that the operator $V_{1,0;0}$  increases $\sin^2\alpha$.
(see Theorem-4.1-v.mw in~\cite{maple} for the details of computations).
\end{proof}

\vspace{2mm}

{
\noindent{\it Proof of Theorem~\ref{greater e}$($v$)$.}
Without loss of generality we assume that $|p_1|<|p_2|$ and as a consequence $q_1<q_2$.
Then in the notation of Theorem~\ref{greater e}~$($v$)$ applied to Proposition~\ref{y-axes} we have:
$$
a=q_1, \quad \hbox{and} \quad k=\frac{q_1}{|p_1|}
$$
From assumption of Theorem~\ref{greater e}$($v$)$ we have
$$
a=q_1\ge\frac{q_1}{|p_1|}+1=k+1
$$
it is also clear that $k>0$. Hence we have all assumptions of Proposition~\ref{fin-3-5}.
Hence, the operator $V_{1,0;0}$  increases $\sin^2\alpha$ for such $\xi$-states.
Therefore, $\Omega_{\max}$ does not contain states satisfying the assumption of Theorem~\ref{greater e}~$($v$)$.
This concludes the proof of Theorem~\ref{greater e}$($v$)$.
\qed}


\section{Proof of Theorem~\ref{greater e}~$($i$)$}\label{teor-i}
We start with several preliminary statements in Subsection~\ref{Several preliminary statements}
We consider Cases~I, II,  and~III in Subsubsections~\ref{CaseI}, \ref{CaseII} and~\ref{CaseIII} respectively.
The proof of Theorem~\ref{greater e}~$($i$)$ is concluded in Subsubsection~\ref{CaseI-III}.

\subsection{Several preliminary statements}
\label{Several preliminary statements}

In this subsection we collect
some important infinitesimal properties of angles between planes for several families of pairs of planes in $\r^3$.

For the next proposition we set the following notation.

$$
\begin{array}{l}
f(b,c,\lambda,t)=(b^2-\lambda^2)t^2+2c\lambda^2t+1-(c^2\lambda^2-\lambda^2);\\
g(b,c,\lambda,t)=
(b^4-b^2\lambda^2+b^2-\lambda^2)t^3+
(-3b^2c+3c\lambda^2)t^2\\
\qquad\qquad \qquad
+(b^2c^2\lambda^2+2b^2c^2+b^2\lambda^2
-3c^2\lambda^2+b^2-\lambda^2-1)t
+c^3\lambda^2+c\lambda^2+c.
\end{array}
$$

\begin{proposition}\label{gamma-derivative}
Let $T(t)$ be a tetrahedron with vertices
$$
P_1(t)=(1,\lambda,1),
\quad
P_2(t)=(1,-\lambda,1),
\quad \hbox{and} \quad
P_3(t)=(0,0,1),
\quad
Q(t)=(t-c,bt,0),
$$
where $\lambda>0$, $b$ and $c$ are arbitrary real number, and the parameter $t\in \r$ $($see Figure~\ref{p.6}$)$.
Denote by $\gamma(b,c,\lambda,t)$ the angle between the planes
$Q(t)P_1(t)P_3(t)$ and $Q(t)P_2(t)P_3(t)$.
Then we have
$$
\sgn\left(\frac{\partial (\sin^2 \gamma)}{\partial t}\right)=-\sgn(fg).
$$
\end{proposition}

\begin{figure}
\centering
 \includegraphics{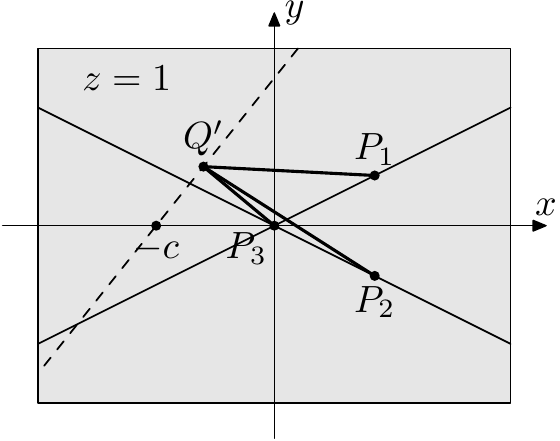}
\caption{The projection of $P_1P_2P_3Q$ to the plane $z=1$ along $z$-axis.
Here $Q'$ is the projection of $Q$, it
evolves along the dashed line with respect to parameter $t$.}
\label{p.6}
\end{figure}

\begin{remark}
Note that the points $P_1$ and $P_2$ can be chosen arbitrarily on the punctured lines parametrically defined by
$\{(u,\pm \lambda u, 1)| u\in \r\setminus \{0\}\}$
respectively. This will not change the values of the angle $\gamma$, which are always in $[0,\pi/2]$.
\end{remark}

\begin{proof}
Since
$$
\sin\gamma=\frac{|(QP_1\times QP_3)\times(QP_2\times QP_3)|}{|QP_1\times QP_3|\cdot|QP_2\times OP_3|},
$$
we have
$$
\begin{array}{l}
\displaystyle
\sin^2\gamma(t)=
\frac {4{\lambda}^{2} \left( {b}^{2}{t}^{2}+{c}^{2}-2\,ct+{t}^{2}+
1 \right) }{ \left( {b}^{2}{t}^{2}+2\,bc\lambda\,t-2\,b\lambda\,{t}^
{2}+{c}^{2}{\lambda}^{2}-2\,c{\lambda}^{2}t+{\lambda}^{2}{t}^{2}+{
\lambda}^{2}+1 \right)}
\\
\displaystyle
\times\frac  {1}{\left( {b}^{2}{t}^{2}-2\,bc\lambda\,t+2\,b
\lambda\,{t}^{2}+{c}^{2}{\lambda}^{2}-2\,c{\lambda}^{2}t+{\lambda}
^{2}{t}^{2}+{\lambda}^{2}+1 \right). }
\end{array}
$$
Denote the denominator of this expression by $h(b,c,\lambda,t)$.
Direct computations show that
$$
\frac{\partial(\sin^2\gamma)}{\partial t}=-8\lambda^2\frac{fg}{h^2}.
$$
Hence the sign of $-fg$ coincides  the sign of $\frac{\partial(\sin^2\gamma)}{\partial t}$. This concludes the proof.
\end{proof}

\begin{corollary}\label{gamma-1-corollary}
For the case $c>0$ let
$$
\left\{
\begin{array}{l}
\lambda\le 1\\
b>\lambda\\
\end{array}
\right.
.
$$
Then for every
$\displaystyle t\in\left[\frac{\lambda c}{b+\lambda},c\right]$
we have
$\displaystyle \frac{\partial (\sin^2 \gamma)}{\partial t}<0$
$($see on Figure~\ref{p.7-8}, left$)$.
\end{corollary}

\begin{figure}
\centering
\includegraphics{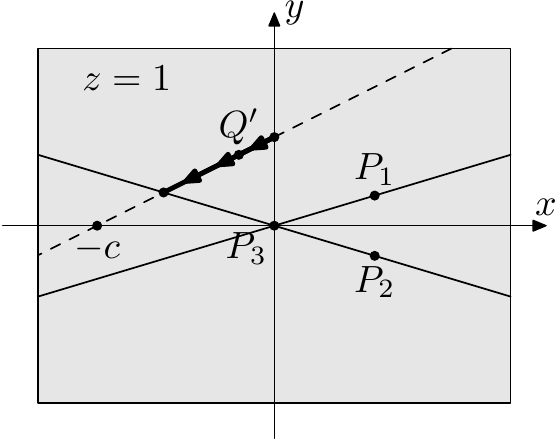}\qquad\qquad\includegraphics{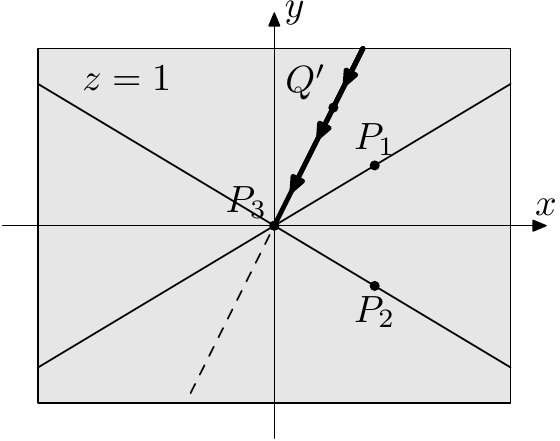}
\caption{Direction of growth for $\gamma$ while projection $Q'$ evolves on plane $z=1$.
Left: case $c>0$. Right: case $c=0$.}
\label{p.7-8}
\end{figure}

\begin{remark}
The value
$$
t=\frac{\lambda c}{b+\lambda}
$$
corresponds to the case $Q'\in P_2P_3$.
The value $t=c$ corresponds to the case when $Q'$ is contained in the $y$-axis, see Figure~\ref{p.7-8}, left.
\end{remark}

\begin{proof}
For the case $c>0$ and $b>\lambda$ the expression for $f$ is a quadratic polynomial in
$t$ with positive coefficients at $t$ and $t^2$:
$$
\left( {b}^{2}-{\lambda}^{2} \right) {t}^{2}+2\,c{\lambda}^{2}t-{c}
^{2}{\lambda}^{2}-{\lambda}^{2}+1.
$$
Hence the minimum of $f$ is attained at $t_0=\frac{-c}{(b^2-\lambda^2)}$;
for $t\ge t_0$ the function $f$ is increasing (in variable $t$).
Both endpoints of the segment $\left[\frac{\lambda c}{b+\lambda},c\right]$
are positive, so the minimum of $f$ on that segment
is attained at the smallest endpoint:
$$
f\Big(\frac{\lambda c}{b+\lambda}\Big)=1-\lambda^2>0.
$$
Therefore, $f$ is positive on $\left[\frac{\lambda c}{b+\lambda},c\right]$.

Let us collect the coefficients of $g$ as follows
$$
\begin{aligned}
g(b,c,\lambda,t)=&
b^4t^3+b^2\lambda^2t+b^2t+\lambda^2b^2t(c^2-t^2)+\lambda^2(c-t)^3\\
&+b^2t(t-c)(t-2c)+(c-t)(1+\lambda^2).
\end{aligned}
$$
It is clear that all the summands of $g(b,c,\lambda,t)$ are nonnegative for $t\in [0,c]$
(while the first one is positive at $t=c$ and the last one is positive at $t=0$),
 and hence $g(b,c,\lambda,t)>0$.
Now the statement of this corollary follows directly from Proposition~\ref{gamma-derivative}.
\end{proof}

\begin{corollary}\label{gamma-2-corollary}
For the case $c=0$ let
$$
\left\{
\begin{array}{l}
\lambda\le 1\\
b\ge 1\\
\end{array}
\right.
.
$$
Then for every real $t\ge 0$ we have  $\displaystyle \frac{\partial\gamma}{\partial t}\le 0$.
The last inequality is exact only in the case $b=\lambda=1$
$($see on Figure~\ref{p.7-8}, right$)$.
\end{corollary}

\begin{proof}
Note that
$$
\begin{array}{l}
f(b,0,\lambda,t)=(b^2-\lambda^2)t^2+1-\lambda^2,\\
g(b,0,\lambda,t)=(b^2t^3+t^3)(b^2-\lambda^2)+t(b^2-1)(1+\lambda^2).
\end{array}
$$
Hence, if $\lambda\le 1$ and $b\ge1$ (excluding the case $b=\lambda=1$) we have $f,g>0$,
and, therefore, $\displaystyle \frac{\partial (\sin^2 \gamma)}{\partial t}< 0$.
Finally, if $b=\lambda=1$, then $\displaystyle \frac{\partial (\sin^2 \gamma)}{\partial t}= 0$.
\end{proof}

\begin{remark}
Note that the identities $b=\lambda=1$ never happen for cubic vectors.
\end{remark}

\begin{corollary}\label{gamma-3-corollary}
For the case $c>0$ let
$$
\left\{
\begin{array}{l}
\lambda\le 1\\
b>\lambda\\
\end{array}
\right.
.
$$
Then for every
$\displaystyle t\in\left(-\infty, -\frac{\lambda c}{b-\lambda}\right]$
the function
$\sin^2(\gamma)$ is increasing while $t$ increase
$($see on Figure~\ref{p.10-11}, left$)$.
\end{corollary}

\begin{figure}
\centering
\includegraphics{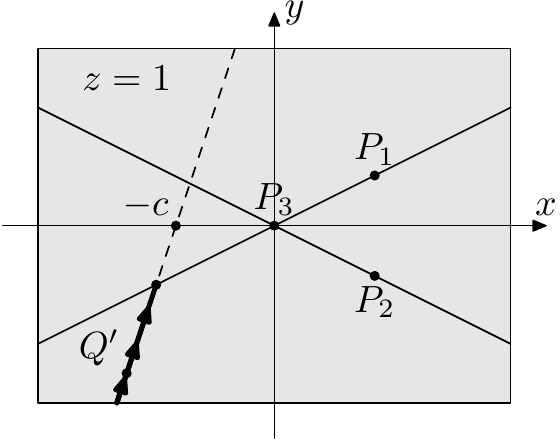} \qquad \qquad \includegraphics{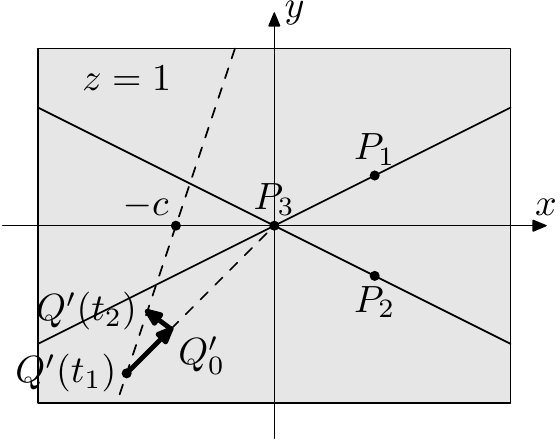}
\caption{Left: direction of growth for $\gamma$ on plane $z=1$.
Right: Reduction of the statement to the above ones.}
\label{p.10-11}
\end{figure}

\begin{proof}
Let
$$
t_1<t_2<-\frac{\lambda c}{b-\lambda}.
$$
Consider the point $Q_0$ on the segment with endpoints $Q(t_1)$ and $O(0,0,0)$
such that $Q_0Q(t_2)$ is orthogonal to $OQ(t_1)$
$($see on Figure~\ref{p.10-11}, right: here the projections of $Q(t_1)$,
$Q(t_2)$ and $Q_0$ to the plane $z=0$ along $z$-axes
are marked as $Q'(t_1)$,
$Q'(t_2)$ and $'Q_0$ respectively$)$.

\vspace{1mm}

By Corollary~\ref{gamma-2-corollary}
the function $\sin^2\gamma$ for $Q_0,P_1(t_1)P_2(t_1)P_3(t_1)$ is smaller than the function $\sin^2\gamma$ for
$Q(t_1)P_1(t_1)P_2(t_1)P_2(t_1)$.

\vspace{1mm}

By Corollary~\ref{gamma-1-corollary}
the function $\sin^2\gamma$ for $Q'(t_2)P_1(t_2)P_2(t_2)P_3(t_2)$ is smaller than the function $\sin^2\gamma$ for
$Q_0'P_1(t_1)P_2(t_1)P_3(t_1)$. (Notice here that $P_i(t_1)=P_i(t_2)$ for $i=1,2,3$.)

\vspace{1mm}

Therefore, the statement of the lemma holds.
\end{proof}

\begin{corollary}\label{gamma-4-corollary}
For the case $Q(t)=(c,t,0)$ $($vertical case$)$ let
$\lambda < 1$.
Then for every
$\displaystyle t\in(-\infty, -\lambda c]$
the function
$\sin^2(\gamma)$ is increasing while $t$ increases.
\end{corollary}

\begin{proof}
The proof repeats the proof for Corollary~\ref{gamma-3-corollary}:
again it is a combination of Corollary~\ref{gamma-1-corollary} and Corollary~\ref{gamma-2-corollary}.

A new case here is when $c=0$. Here the expression for $\sin^2(\gamma)$
is as follows:
$$
4\,{\frac {{\lambda}^{2} \left( {t}^{2}+1 \right) }{ \left( {\lambda
}^{2}+{t}^{2}+1 \right) ^{2}}}.
$$
Its derivative
$$
8\,{\frac {{\lambda}^{2}t \left( {\lambda}^{2}-{t}^{2}-1 \right) }{
 \left( {\lambda}^{2}+{t}^{2}+1 \right) ^{3}}}
$$
is negative as long as $\lambda\le 1$ and $t<0$ .
Hence the value for $\sin^2(\gamma)$ increases.
\end{proof}

Computations for main statements in this section are in  Subsection-6.1.mw in~\cite{maple}.

\subsection{Case I}\label{CaseI}
Let two sides of the triangle $T(s)$ intersect the positive quadrant at $y$-coordinate ray.
Denote these points by $M(0,m,1)$ and $N(0,n,1)$ where $m>n\ge 0$.

\begin{lemma}\label{lemma-case1}
For every positive constants $C_1$, $C_2$ and $C_3$,
the number of elements in $\Omega_{\max}(\xi;\nu_1,\nu_2)$  satisfying
the system
$$
\left\{
\begin{array}{l}
m<C_1\\
y_0<x_0<C_2\\
m-n>C_3
\end{array}
\right.
$$
is finite.
\end{lemma}

\begin{proof}
Since $m-n>C_3$, we have
$$
\area (\xi,M,N)>\frac{C_3}{2}.
$$

From conditions of proposition we have boundedness of
$m$, $n$, $x_0$, $y_0$.
By Proposition~\ref{frozen-vertex}$($ii$)$ the area of $T$ is bounded from above, say by $A_0$.
Therefore, from all the above the vectors  $\xi\nu_1$ and $\xi\nu_2$ are at most
$$
\frac{\area (T)}{\area (\xi,M,N)}<\frac{2A_0}{C_3}
$$
times greater than $\xi M$ and $\xi N$.
So the coordinates $x_1$, $y_1$, $x_2$, and $y_2$ are bounded.
Therefore all the coordinates of the triangle $T$ are bounded from above.

\vspace{1mm}

Secondly, we have
$$
\vol\big(\pyr(C,\pi)\big)>\frac{1}{3}\area(T)>\frac{1}{3}\area (\xi,M,N)>\frac{C_3}{6}.
$$
By Proposition~\ref{finiteness-unit-distance-planes-2}
the number of integer planes $\pi$ on the unit integer distance
to the origin satisfying
$$
\vol\big(\pyr(C,\pi)\big)>\frac{C_3}{6}
$$
is finite up to the action of the positive Dirichlet group $\Xi_+(C)$.
Finally by Proposition~\ref{bounded-triangle}
there are finitely many coordinate choices for triangles with bounded coordinates.
Each of such choices corresponds to at most one separating $\xi$-state.

Therefore, the number of separating $\xi$-states satisfying the conditions of the proposition is finite.
\end{proof}

\begin{corollary}\label{case1-corollary}
For every positive real number $\varepsilon$ $($in the conditions of Case~I$)$.
There are only finitely many elements in $\Omega_{\max}(\xi;\nu_1,\nu_2)$
satisfying $|p_1-p_2|\ge \varepsilon$.
\end{corollary}

\begin{proof}
In the conditions Case~I we have
$$
|p_1-p_2|x_0=(m-n)x_0=2 \area(\tilde\xi MN)<2\area(T).
$$
By Corollary~\ref{S(T)-finite} the area of $\area(T)$ is bounded from above (recall that $\pi$-area
of $T$ coincides with the area of $T$ for the case of the standard basis). Hence, $x_0$ and $|p_1-p_2|$
are both bounded from above.
Therefore, $y_0<x_0$ is also bounded from above.

\vspace{1mm}

Let $y_0$ and $|p_1-p_2|$ are bounded, say  by some constant $K$. Then for $p_1>2K$
the angle $\angle P_1\xi P_2$ is less than $\pi/2$ (as $p_1>y$ and $p_2>y$), and hence we are in position to apply Corollary~\ref{gamma-1-corollary}
either to operator $V_{0,1;0}$ if the bisectrix of the angle  $\angle P_1\xi P_2$ intersects
the $y$-axes with the angle greater or equal to $\pi/4$ or $V_{1,0;0}$ otherwise. Hence such state is not in
$\Omega_{\max}(\xi;\nu_1,\nu_2)$.

Therefore,  if the state is in $\Omega_{\max}(\xi;\nu_1,\nu_2)$
then both $p_1$ and $p_2$ (and hence $m$ and $n$) are bounded from above by some constant.
Recall that $x_0$ and $y_0$ are also bounded, and by the assumption of the lemma we have the inequality
$$
m-n=|p_1-p_2|\ge \varepsilon.
$$
Then we arrive to the situation of Lemma~\ref{lemma-case1}. Therefore, there are finitely many
separating $\xi$-states satisfying these conditions.
These are the only $\xi$-states that might be in $\Omega_{\max}(\xi;\nu_1,\nu_2)$. That concludes the proof.
\end{proof}


\subsection{Case II}\label{CaseII}
Let two sides of the triangle $T(s)$ intersect the positive quadrant one at $x$-coordinate ray
and the second at $y$-coordinate ray.
Denote these points by $M(0,m,1)$ and $N(n,0,1)$ where $m,n>0$.

\begin{lemma}\label{lemma-case2}
For every positive constants $C_1$, $C_2$, $C_3$, and $C_4$
the number of separating $\xi$-states in $\Omega_{\max}(\xi;\nu_1,\nu_2)$ satisfying
the system
$$
\left\{
\begin{array}{l}
n<C_1\\
m<C_2\\
y_0<x_0<C_3\\
\area(\tilde\xi,P,Q)>C_4
\end{array}
\right.
$$
is finite.
\end{lemma}

\begin{proof}
By Proposition~\ref{finiteness-unit-distance-planes-2}
the number of integer planes $\pi$ on the unit integer distance
to the origin satisfying
$$
\vol\big(\pyr(C,\pi)\big)>\frac{1}{3}\cdot \area(\tilde\xi, M,N)>\frac{C_4}{3}
$$
is finite up to the action of the positive Dirichlet group $\Xi_+(C)$.

Now the coordinates $x_0$, $y_0$, $m$, and $n$ are bounded.
The coordinates $x_i$ and $y_i$ for $i=1,2$ are bounded by
$$
C_3+C_3\frac{\area (T)}{\area (\tilde\xi,M,N)}<C_3+\frac{\area (T)}{C_4}.
$$
Since $\area(T)$ is uniformly bounded for algebraic cones (Proposition~\ref{frozen-vertex}$($ii$)$),
the coordinates $x_i$ and $y_i$ for $i=1,2$ are bounded.

Further by Proposition~\ref{bounded-triangle}
there are finitely many coordinate choices for triangles with bounded coordinates.
Each of such choices corresponds to at most one separating $\xi$-state.

Therefore, the number of  separating $\xi$-states satisfying the conditions of the proposition is finite.
\end{proof}

\begin{corollary}\label{case2-corollary}
For every positive real number $\varepsilon$ $($in the conditions of Case~II$)$.
There are only finitely many elements in $\Omega_{\max}(\xi;\nu_1,\nu_2)$
satisfying
$$
\min(|p_2-p_1|,|q_2-q_1|)\ge \varepsilon.
$$
\end{corollary}

\begin{proof}
Denote $O'=(0,0,1)$ and set
$$
\begin{array}{l}
\beta=\angle \xi M O',\\
\gamma=\angle Q_1\xi Q_2.
\end{array}
$$

\vspace{2mm}

{\noindent{\bf Uniform boundedness of $|x_0-n|$ and  $|y_0-m|$.}}
First of all we note that
$$
\area(T)>\area(M\xi N)>\frac{|x_0-n|}{n}\cdot\area(NP_1P_2)=\frac{|x_0-n|}{n}\cdot\frac{n|P_1P_2|}{2}>\frac{|x_0-n|\varepsilon}{2}.
$$
(this is a proportion in which $\xi$ divides the segment $NP_i$ where $P_i\ne M$).
Since $\area(T)$ is uniformly bounded from above (Proposition~\ref{frozen-vertex}$($ii$)$),
the quantity $|x_0-n|$ is uniformly bounded from above.

\vspace{1mm}

The proof of the uniform boundedness $|y_0-m|$ repeats the proof for  $|x_0-n|$ (one should swap the coordinates).

\vspace{2mm}

{\noindent{\bf Uniform boundedness of $y_0$.}}
If $y_0$ is large enough (and hence $x_0>y_0$ is large), then due to previous item, the angle $\angle M\xi N$ is close to $\pi/2$.
Hence $\area(T)$ should be sufficiently large as well.
Therefore,
the uniform boundedness of $\area(T)$  (Proposition~\ref{frozen-vertex}$($ii$)$)
implies the  uniform boundedness of $y_0$.

\vspace{2mm}

{\noindent{\bf Uniform boundedness of $\sin \gamma$ by a positive constant.}}
Consider the triangle $\xi P_1 P_2$. Without loss of generality we assume that $M=P_1$.
Then according to the sine rule we have:
$$
\frac{\sin\gamma}{|p_1-p_2|}=\frac{\sin\beta}{|P_2\xi|}.
$$
Note that $p_1-p_2>\varepsilon$ and $|P_2\xi|$
Therefore,
$$
\sin\gamma>\varepsilon\sin\beta.
$$
It remains to note that
$$
\tan\beta=\frac{x_0}{|y_0-m|}>\frac{1}{|y_0-m|}
$$
is uniformly bounded from below.
Hence $\sin\beta$ is uniformly bounded from below by a positive constant.
Therefore, $\sin\gamma$ is uniformly bounded from below by a positive constant.

\vspace{2mm}

{\noindent{\bf Case of $\gamma$ bounded from below by a positive constant.}}
First of all let us show that for every $\delta>0$,
there are only finitely many elements in $\Omega_{\max}(\xi;\nu_1,\nu_2)$
satisfying
$$
\sin\gamma>\delta.
$$
Note that
$$
\area(T)>\area(M\xi N)=\frac{|M\xi||N\xi|\sin\angle M\xi N}{2}=
\frac{|M\xi||N\xi|\sin\gamma}{2}.
$$
First,
since the $\area(T)$ is uniformly bounded from above  (Proposition~\ref{frozen-vertex}$($ii$)$) and $\sin\gamma>\delta$,
we get that the product $|M\xi||N\xi|$ is uniformly bounded from above.
Since the factors are both greater than 1, we have uniform boundedness of
$m,n,x_0,y_0$.

Secondly,
$$
\area(T)=\frac{|M\xi||N\xi|\sin\gamma}{2}\ge \frac{\delta}{2}.
$$
At this point we have all the assumptions of Lemma~\ref{lemma-case2}.
Hence the are only finitely many elements in $\Omega_{\max}(\xi;\nu_1,\nu_2)$
satisfying
$$
\sin\gamma>\delta.
$$

{\noindent{\bf Conclusion of the proof.}}
First, we have shown that $\sin \gamma$ is uniformly bounded from below by a positive constant.
Secondly, for every $\delta>0$ we have proved finiteness of
elements in $\Omega_{\max}(\xi;\nu_1,\nu_2)$
satisfying
$$
\sin\gamma>\delta.
$$
This concludes the proof of the corollary.
\end{proof}


\subsection{Case III}\label{CaseIII}
Let two sides of the triangle $T(s)$ intersect the positive quadrant with $x$-coordinate ray.
Denote these points by $M(m,0,1)$ and $N(n,0,1)$ where $n>m>0$.

\begin{lemma}\label{lemma-case3}
For every positive constants $C_1$, $C_2$ and $C_3$,
the number of  separating $\xi$-states in $\Omega_{\max}(\xi;\nu_1,\nu_2)$ satisfying
the system
$$
\left\{
\begin{array}{l}
n<C_1\\
y_0<x_0<C_2\\
n-m>C_3
\end{array}
\right.
$$
is finite.
\qed
\end{lemma}

The proof of Lema~\ref{lemma-case3} is similar to the proof of Lema~\ref{lemma-case1}
and it is omitted here.

\begin{corollary}\label{case3-corollary}
For every positive real number $\varepsilon$ $($in the conditions of Case~III$)$.
There are only finitely many elements in $\Omega_{\max}(\xi;\nu_1,\nu_2)$
satisfying the condition
$$
\min(|p_2-p_1|,|q_2-q_1|)\ge \varepsilon.
$$
\end{corollary}

\begin{proof}
In Case~III we have
$$
|q_1-q_2|y_0=(n-m)y_0=2 \area(\tilde\xi MN)<2\area(T).
$$
By Corollary~\ref{S(T)-finite} the area of $\area(T)$ is bounded from above (recall that $\pi$-area
of $T$ coincides with the area of $T$ for the case of the standard basis). Hence, $y_0$ and $|q_1-q_2|$
are both bounded from above.
Therefore, $y_0<x_0$ is also bounded from above.

\vspace{1mm}

Now we separately consider several cases of the position of  points $P_1$ and $P_2$.
Note that if $p_i\le y_0$ then $p_i\le 0$ since  we have $q_i\ge 0$  for Case~III.

\vspace{2mm}

{\underline{\bf Case $p_1>y_0$ and $p_2>y_0$.}}
Here the angle $\angle P_1\xi P_2$ is less than $\pi/2$ (as $p_1>y$ and $p_2>y$), and hence we are in position to apply Corollary~\ref{gamma-1-corollary}
either to operator $V_{0,1;0}$ if the bisectrix of the angle  $\angle P_1\xi P_2$ intersects
the $y$-axes with the angle greater or equal to $\pi/4$ or $V_{1,0;0}$ otherwise. Hence all such states are not in
$\Omega_{\max}(\xi;\nu_1,\nu_2)$.

\vspace{2mm}

{\underline{\bf Case $p_1>y_0$ and $-2\le p_2\le 0$.}}
In this case
$$
\min(q_1,q_2)<\frac{x_0}{2} \quad \hbox{and}
\quad \max(q_1,q_2)>x_0.
$$
Hence $x_0$ is bounded, and then $q_1$ and $q_2$ are bounded.
We arrive at conditions of Lemma~\ref{lemma-case3}, hence
there is only  finitely many elements in $\Omega_{\max}(\xi;\nu_1,\nu_2)$ in this case.

\vspace{2mm}

{\underline{\bf Case $p_1>y_0$ and $p_2< -2$.}}
Denote
$$
c=\min(q_1,q_2)\quad \hbox{and}
\quad
k=\frac{c}{p_1}
$$
Here we have $c\ge k+1>0$ and, therefore, $V_{1,0;0}$ increases $\sin^2(\alpha)$ by Proposition~\ref{fin-3-5}.

\vspace{2mm}

{\underline{\bf Case $p_2>y_0$ and $p_1< 0$.}} This case is a repetition of the previous two cases with $p_1$ and $p_2$ swapped.

\vspace{2mm}

{\underline{\bf Case $0\ge p_1\ge -1.05$ and $p_2<p_1$.}}
First of all the condition $p_1>p_2$ implies $q_1<q_2$.
Denote by $D$ the intersection of the line $\xi P_2$ with the line parallel to the $y$-axes and passing through $P_1$.
Then
$$
\frac{|Q_1Q_2|}{|DP_1|}=\frac{y_0}{y_0+|p_2|}=\frac{1}{1+|p_2|/y_0}>\frac{1}{1+21/20}=\frac{20}{41}.
$$
Hence $DP_1$ is bounded from above by some constant, say $K$.
Since $|P_1P_2|$ is bounded from below, the ratio
$$
\frac{|P_1P_2|}{|DP_1|}
$$
is bounded from above.

Therefore, the coordinates of $D(d_1,d_2,1)$ satisfy
$$
d_1\le K , \quad
d_2\ge -1.05,
$$
and the slope of $\xi P_2$ is bounded from below.
Hence  $q_2$ is bounded (and hence $q_1$ is bounded).
Finally the boundedness of $y_0$ together with the above imply the boundedness of $x_0$.
Therefore, we have all the conditions of Lemma~\ref{lemma-case3}, hence
there is only  finitely many elements in $\Omega_{\max}(\xi;\nu_1,\nu_2)$ in this case.

\vspace{2mm}

{\underline{\bf Case $0\ge p_1\ge -1.05$ and $p_1<p_2$.}} This case is a repetition of the previous one with $p_1$ and $p_2$ swapped.

\vspace{2mm}

{\underline{\bf Case $p_1,p_2\ge -1.05$.} }
Without loss of generality we assume that $p_1>p_2$ (and hence $q_1<q_2$).
Denote
$$
c=q_2\quad \hbox{and}
\quad
k=\frac{q_2}{|p_1|}.
$$

First of all we study the case $c\ge 21$. Here we have
$$
c-k=\bigg(1-\frac{1}{|p_1|}\bigg)\cdot c\ge \bigg(1-\frac{1}{1.05}\bigg)\cdot 21=1.
$$
Hence $c\ge k+1>0$. Now we have all the conditions of Proposition~\ref{fin-3-5}.
Therefore, $V_{1,0;0}$ increases the value of $\sin^2(\alpha)$.

Now let $c<21$.
Hence
$$
k\le \frac{c}{|p_1|}\le \frac{21}{1.05}.
$$
Since $y_0$, $c$, and $k$ are bounded, $x_0$ is bounded as well.
Further since $p_1,p_2<1$ then $q_1,q_2<c$.
Therefore, we have all the conditions of Lemma~\ref{lemma-case3}, hence
there is only  finitely many elements in $\Omega_{\max}(\xi;\nu_1,\nu_2)$ in case of $p_1,p_2\ge -1.05$.

This concludes the proof of Corollary~\ref{case3-corollary}.
\end{proof}


\subsection{Conclusion of the proof of Theorem~\ref{greater e}~$($i$)$}\label{CaseI-III}
The proof for all possible cases is given in Corollaries~\ref{case1-corollary},
\ref{case2-corollary} and~\ref{case3-corollary}.
\qed

\section{Proof of Theorem~\ref{greater e}~$($ii$)$}\label{teor-ii}
We start with the proof with Case~II. Then we continue with Cases~I and~III.

{\noindent {\bf Case II.}}
Let $|p_2-p_1|<\varepsilon$, hence
$$
\min(p_1,p_2)\ge\min(m,m-\varepsilon)\ge-\varepsilon>0.1.
$$
Similarly if $|q_2-p_1|<\varepsilon$, hence
$$
\min(q_1,q_2)\ge\min(n,n-\varepsilon)\ge-\varepsilon>0.1.
$$

\vspace{2mm}

{\noindent {\bf Case I.}}
If $q_1>0$ or $q_2>0$ then the proof is similar to the one in Case~II.

Let now both $q_1,q_2<0$.
Since $y_0<x_0<1$ we have
$$
|p_2-p_1|<|q_2-q_1|,
$$
and hence $|p_2-p_1|<\varepsilon$.

\vspace{2mm}

{\noindent {\bf Case III.}}
If $p_1>-1.05$ or $p_2>-1.05$ then the proof is similar to the one in Case~II.

Let finally $p_1,p_2<-1.05$.  In case if $|q_1-q_2|<\epsilon$ we are done.
Assume now that $|p_1-p_2|<\epsilon$.
Without loss of generality we assume that $p_1>p_2$ and, therefore, $q_1<q_2$.

First, let
$$
\frac{q_1}{|p_1|}<20,
$$
then
$$
|q_1-q_2|<20 |p_1-p_2|<20\varepsilon.
$$
Secondly, assume that
$$
\frac{q_1}{|p_1|}\ge  20,
$$
then
$$
q_1=|p_1|\frac{q_1}{|p_1|}\ge 1.05\frac{q_1}{|p_1|}= \frac{q_1}{|p_1|}+ 0.05\frac{q_1}{|p_1|}\ge \frac{q_1}{|p_1|}+1.
$$
This concludes the proof of Theorem~\ref{greater e}~$($ii$)$ for Case~III.
\qed


\section{Proof of Theorem~\ref{greater e}~$($iii$)$ and~$($iv$)$}\label{theor-iii-iv}

In this section we prove  Theorem~\ref{greater e}~$($iii$)$ and~$($iv$)$.
We start in Subsections~\ref{p-asym} and~\ref{q-asym} with some notation and settings
that we further need in the proof.
Additionally we study asymptotic for
for admissible states for small angles at $\xi$ (Propositions~\ref{y-axes} and~\ref{x-axes}).
The set of admissible states is actually covered by 19 cases.
We study all these cases and therefore prove Propositions~\ref{y-axes} and~\ref{x-axes} in Subsection~\ref{cases1-19}.
This will imply the  statements of Theorem~\ref{greater e}~$($iii$)$ and~$($iv$)$.

\subsection{Families $p$-asymptotically increasing $\sin^2\alpha$}\label{p-asym}

Let us first introduce the following notation.
For a linear operator $M$ in $\r^3$ and a state $s$ with
$$
\xi=(x,y,1), \qquad P_1=(0,p_1,1), \quad   \hbox{and} \quad P_2=(0,p_2,1).
$$
(here we take $P_1$ and $P_2$ as in preliminary set-up in Section~\ref{Theorem-ii}).
Let us set
$$
F_M(p_1,p_2,x,y)=\frac{\sin^2\alpha(M(s))}{\sin^2\alpha(s)}-1;
$$
and let us expand by linearity
$$
F_M(p,p,x,y)=\lim\limits_{\varepsilon\to 0} F_M(p,p+\varepsilon, x,y).
$$
\begin{remark}\label{formulaF}
Note that $\sin^2\alpha(s)$ has a nice expression in terms of $p_1,p_2,x,y$.
$$
\sin^2\alpha(s)={\frac {{\left( p_1-p_2 \right) ^{2} \left( {x}^{2}+{y}^{2}+1
 \right) x}^{2} }{ \left( {x}^{2}+{x}^{2}{p_1}^{2}+(y-p_1)^{2}
 \right)  \left( {x}^{2}+{x}^{2}{p_2}^{2}+(y-p_2)^{2}
 \right)
}}
.
$$
Note also that if $p_1=p$ and $p_2=p+\varepsilon$ then we have
$$
\sin^2\alpha(s)={\frac {{x}^{2} \left( {x}^{2}+{y}^{2}+1 \right) }{ \left( {x}^{2}{
p}^{2}+{x}^{2}+{(y-p)}^{2} \right) ^{2}}}
\varepsilon^2 + O(\varepsilon^3).
$$
\end{remark}

Let us introduce the following two important definitions.

\begin{definition}
We say that a set of operators $\M$ {\it $p$-max-asymptotically increases} $\sin^2\alpha$
on a domain $D$ if
there exists $\varepsilon>0$ such that for any
$(p,x,y)\in D$ there exists $M\in \M$
such that for any positive $\delta<\varepsilon$ we have
and any $(p,x,y)\in D$ we have
$$
F_M(p,p+\delta,x,y)>0.
$$
\end{definition}

\begin{definition}
We say that a set of operators $\M$ {\it $p$-min-asymptotically increases} $\sin^2\alpha$
on a domain $D$ if
there exists $\varepsilon>0$ such that for any positive $\delta<\varepsilon$, any $M\in \M$
and any $(p,x,y)\in D$ we have
$$
F_M(p,p+\delta,x,y)>0.
$$
\end{definition}

Let us write a criterium for  the family $\M$ to increase $p$-min-asymptotically the function $\sin^2\alpha$
on $D$.

\begin{proposition}\label{prop-a-1}
Let $\M$ be a family of operators continuously depending on $d$ parameters $($$d\ge 0$$)$;
each of the parameters is defined on a segment $[0,1]$.
Consider a domain  $D$ of triples $(p,x,y)$ in $\r^3$. Denote by $\bar D$ its closure in $\r P^3$.
Assume that for any $M\in \M$ we have the following:
\begin{itemize}
\item $F_M(p,p,x,y)>0$ on $\bar D$ $($here and below we extend $F_M$ by continuity to $\r P^3$$)$;
\item $F_M(p,p,x,y)$ is defined on $\bar D$ $($i.e., has no zeroes of the polynomials in the denominator in appropriate coordinates$)$.
\end{itemize}
Then the family $\M$ $p$-min-asymptotically increases $\sin^2\alpha$
on $D$.
\end{proposition}

\begin{remark}
In case of  $d=0$, the family $\M$ consists of a single operator.
\end{remark}

\begin{proof}
First set
$$
G_M(p,x,y)=F_M(p,p,x,y).
$$

Let $M\in \M$.
Since $G_M(p,x,y)>0$ on $\bar D$, there exists a positive real number $\tau_1(M)$ such that
$$
\inf(G_M)\ge \tau_1(M).
$$
Since the space of parameters is compact, one can pick $ \tau_1(\M)>0$ simultaneously for
all operators in $\M$.

\vspace{1mm}

Let $M\in \M$. Since $F_M(p,p,x,y)$ is defined  on $\bar D$, there exist positive real numbers $\tau_2(M)$ and $N(M)$ such that
the absolute value of $\frac{\partial F_M(p,p+\tau,x,y)}{\partial \tau}$ is uniformly bounded by $N(M)$ from above on the set
$$
D\times[0,\tau_2(M)].
$$
Once again since the space of parameters is compact, one can pick $ \tau_2(\M)>0$
and $N(\M)>0$ simultaneously for all operators in $\M$.

\vspace{1mm}

Therefore, $\M$ $p$-min-asymptotically increases $\sin^2\alpha$ on $D$.
Here
$\varepsilon$ can be chosen to be
$$
\min\Big(\frac{\tau_1(\M)}{N(\M)},\tau_2(\M)\Big).
$$
\end{proof}

We have a similar criterium for the family $\M$ to increase $p$-max-asymptotically the function $\sin^2\alpha$
on $D$.

\begin{proposition}\label{prop-a-2}
Let $\M$ be a finite set of operators.
Consider a domain  $D$ of triples $(p,x,y)$ in $\r^3$. Denote by $\bar D$ its closure in $\r P^3$.
Assume that for any  $M\in \M$ we have
\begin{itemize}
\item $F_{M}(p,p,x,y)$ is defined on $\bar D$ $($here we extend $F_M$ by continuity to $\r P^3$$)$.
\item $\max\limits_{M\in\M} F_{M}(p,p,x,y)$ is positive on $\bar D$.
\end{itemize}
Then the family $\M$ $p$-max-asymptotically increases $\sin^2\alpha$
on $D$.
\end{proposition}

\begin{proof}
First set
$$
G_M(p,x,y)=F_M(p,p,x,y).
$$
Since $\max\limits_{M\in\M} G_{M}$ is positive  on $\bar D$ and $\M$ is finite there exists a positive real number $\tau_1$ such that
$$
\inf(\max\limits_{M\in\M} G_{M})\ge \tau_1.
$$

Let $M\in \M$. Since $F_M(p,p,x,y)$ is defined  on $\bar D$, there exist positive real numbers $\tau_2(M)$ and $N(M)$ such that
the absolute value of $\frac{\partial F_M(p,p+\tau,x,y)}{\partial \tau}$ is uniformly bounded by $N(M)$ from above on the set
$$
D\times[0,\tau_2(M)].
$$
Once again since the space of parameters is compact, one can pick $ \tau_2(\M)>0$
and $N(\M)>0$ simultaneously for all operators in $\M$.

\vspace{1mm}

Therefore, $\M$ $p$-max-asymptotically increases $\sin^2\alpha$ on $D$.
Here
$\varepsilon$ can be chosen to be
$$
\min\Big(\frac{\tau_1(\M)}{N(\M)},\tau_2(\M)\Big).
$$
\end{proof}

Now we are ready to formulate the main statement of this subsection.

\begin{proposition}\label{y-axes}
Let
$$
D'=\{(p,x,y){|\,}  p>-1.1,\,  x>y>1\},
$$
and $\M$ be the set of all admissible operators.
Then $\M$ $p$-max-asymptotically increases $\sin^2\alpha$ on $D'$.
\end{proposition}

\begin{proof}

We prove the statement separately for the following subsets $D_i$ of $D'$ and the following
sets $\M_i$. The parametrisation of $D_i$ in all the cases is either $p=a$, $x$, $y$
with parameters $(a,x,y)$
or $p=a$, $x=kt$, $y=a+kt$ with parameters $(a,k,t)$.

\begin{itemize}
\item Case 1 (see Lemma~\ref{inf-1-1}):
$\M_1$ $p$-max-asymptotically increases $\sin^2\alpha$ on $D_1$ where
$$
\begin{array}{l}
\M_1=\{V_{1,0;0},V_{0,1;0}\};\\
D_1=\{(a,x,y){|\,} a\ge 1, x>y>1\}.\\
\end{array}
$$

\item Case 2 (see Lemma~\ref{inf-1-2}):
$\M_2$ $p$-max-asymptotically increases $\sin^2\alpha$ on $D_2$ where
$$
\begin{array}{l}
\M_2=\{ V_{0,1;1}\};\\
D_2=\{(a,kt,a+t){|\,}1\ge a\ge 0.52,k>1,t>1\}.\\
\end{array}
$$

\item Case 3 (see Lemma~\ref{inf-1-3}):
$\M_3$ $p$-min-asymptotically increases $\sin^2\alpha$ on $D_{3,1}$, $D_{3,2}$, and $D_{3,3}$, where
$$
\begin{array}{l}
\M_3=\{V_{0,0;k/3-\varepsilon}{|\,}1> \varepsilon\ge 1\};\\
D_{3,1}=\{(a,kt,a+t){|\,}0.33\ge a\ge 0,k\ge 3.1,t\ge 0.67\};\\
D_{3,2}=\{(a,kt,a+t){|\,}0.52\ge a\ge 0,k\ge 3.1,t\ge 7/4\};\\
D_{3,3}=\{(a,kt,a+t){|\,}0.1\ge a\ge -0.81, 1,k\ge 3.1,t\ge 0.9\}\\
\end{array}
$$

\item Case 4 (see Lemma~\ref{inf-1-4}):
$\M_4$ $p$-min-asymptotically increases $\sin^2\alpha$ on $D_4$ where
$$
\begin{array}{l}
\M_4=\{V_{(1-a)k-\delta ,1;k-\varepsilon-1}{|\,}1> \varepsilon\ge 0,1> \delta\ge 0\};\\
D_{4}=\{(a,kt,a+t){|\,}0.68\ge a\ge 0.3,1001\ge k\ge 6.3,a+t>1\}.\\
\end{array}
$$

\item Case 5 (see Lemma~\ref{inf-1-5}):
$\M_5$ $p$-min-asymptotically increases $\sin^2\alpha$ on $D_5$ where
$$
\begin{array}{l}
\M_5=\{V_{0 ,1;k/2-\varepsilon}{|\,}1> \varepsilon\ge 0\};\\
D_{5}=\{(a,kt,a+t){|\,}0.68\ge a\ge 0.3,k\ge 1000,7/4\ge t\ge 0.31
\}.\\
\end{array}
$$

\item Case 6 (see Lemma~\ref{inf-1-6}):
$\M_6$ $p$-max-asymptotically increases $\sin^2\alpha$ on $D_6$ where
$$
\begin{array}{l}
\M_6=\{V_{1,1;3}\};\\
D_{6}=\{(a,kt,a+t){|\,}0.68\ge a\ge 0.3,6.3\ge k\ge 3.13, a+t>1\}.\\
\end{array}
$$

\item Case 7 (see Lemma~\ref{inf-1-7}):
$\M_{7,1}$ $p$-max-asymptotically increases $\sin^2\alpha$ on $D_{7,1}$, $D_{7,2}$, $D_{7,4}$, and $D_{7,5}$,
and $\M_{7,2}$ $p$-max-asymptotically increases $\sin^2\alpha$ on $D_{7,2}$, where
$$
\begin{array}{l}
\M_{7,1}=\{V_{0,0;1}\} \quad \hbox{and} \quad \M_{7,2}=\{V_{0,1;1}\};\\
D_{7,1}=\{(a,kt,a+t){|\,}0.68\ge a\ge 0,  2\ge k\ge 1,  t\ge 0.31\};\\
D_{7,2}=\{(a,kt,a+t){|\,}0.68\ge a\ge 0,  3\ge k\ge 2,   t\ge 0.31\};\\
D_{7,3}=\{(a,kt,a+t){|\,}0.7\ge a\ge 0.5,  3.3\ge k\ge 2.9  ,t\ge 0.29\};\\
D_{7,4}=\{(a,kt,a+t){|\,}0.51\ge a\ge 0, 3.3\ge k\ge 2.9,  t\ge 0.4\};\\
D_{7,5}=\{(a,kt,a+t){|\,}0\ge a\ge -0.81, 3.1\ge k\ge 1,  t\ge 1\}\}.
\end{array}
$$

\item Case 8 (see Lemma~\ref{inf-1-8}):
$\M_8$ $p$-max-asymptotically increases $\sin^2\alpha$ on $D_8$ where
$$
\begin{array}{l}
\M_8=\{V_{0,1;0}\};\\
D_{8}=\{(a,x,y){|\,}1\ge a\ge 0.6799,2\ge y> 1, x>y\}.\\
\end{array}
$$

\item Case 9 (see Lemma~\ref{inf-1-9}):
$\M_9$ $p$-min-asymptotically increases $\sin^2\alpha$ on $D_9$ where
$$
\begin{array}{l}
\M_9=\{V_{(-a)k-\delta ,1;k/3 -\varepsilon -1}{|\,}1> \varepsilon\ge 0,1> \delta \ge 0\};\\
D_{9}=\{(a,kt,a+t){|\,}-0.8\ge a\ge -1.1,  k\ge 6,  a+t>1\}.\\
\end{array}
$$

\item Case 10 (see Lemma~\ref{inf-1-10}):
$\M_{10}$ $p$-max-asymptotically increases $\sin^2\alpha$ on $D_{10}$ where
$$
\begin{array}{l}
\M_{10}=\{V_{1,0;1}\};\\
D_{10}=\{(a,kt,a+t){|\,}-0.8\ge a\ge -1.1,  6\ge k\ge 1,  a+t>1\}.\\
\end{array}
$$
\end{itemize}

\begin{remark}\label{flooor}
Let us now give a small remark regarding the floor operator.
In Case 3 we have proven that  the family
$$
\M_3=\{V_{0,0;k/3-\varepsilon}{|\,}1> \varepsilon\ge 1\}
$$
$p$-min-asymptotically increases $\sin^2\alpha$ on the domains.
The parameter $\varepsilon$ here is to secure that
the integer operator
$$
V_{0,0;\lfloor k/3\rfloor} \in \M_3.
$$
Hence at each point of the domain we have an integer admissible  operator of the family.
The $p$-min-asymptotically increase of $\sin^2\alpha$ imply that
all such operators increase  $\sin^2\alpha$ by some quantity greater than some uniform positive constant.
(We have a similar situation in Cases 4, 5, and 9.)
\end{remark}

Let us make a brief outline of the cases with respect to a parameter $p=a$:
\begin{itemize}
\item $p\ge 1$ is covered by $D_1$;

\item $1\ge p\ge 0.68$ is covered by $D_2\cup D_8$;

\item $0.68\ge p\ge 0.3$  is covered by $(D_2\cup D_{3,2} \cup D_5)\cup D_4\cup D_6\cup (D_{7,3} \cup D_{7,4})\cup D_{7,2}\cup D_{7,1}$
(here we put domains in the order where $k$ decreases);

\item $0.3\ge p\ge 0$  is covered by $D_{3,1}\cup D_{7,4}\cup D_{7,2}\cup D_{7,1}$
(here we put domains in the order where $k$ decreases);

\item $0\ge p \ge -0.8$ is covered by $D_{7,5}\cup D_{3,3}$.

\item $-0.8\ge p \ge -1.1$ is covered by $D_9\cup D_{10}$.

\end{itemize}

Therefore,  $\M$ $p$-max-asymptotically increases $\sin^2\alpha$ on $D'$.
\end{proof}

\vspace{2mm}

{
\noindent{\it Proof of Theorem~\ref{greater e}~$($iii$)$.}
The statement of Theorem~\ref{greater e}~$($iii$)$ is equivalent to
the statement of Proposition~\ref{y-axes}.
}


\subsection{ Families $q$-asymptotically increasing $\sin^2\alpha$ }\label{q-asym}

Similar to Section~\ref{p-asym} we use the following notation.
For a linear operator $M$ in $\r^3$ and a state $s$ with
$$
\xi=(x,y,1), \qquad Q_1=(q_1,0,1), \quad   \hbox{and} \quad Q_2=(q_2,0,1).
$$
(here we take $Q_1$ and $Q_2$ as in preliminary set-up in Section~\ref{Theorem-ii}).
Let us set
$$
\hat F_M(q_1,q_2,x,y)=\frac{\sin^2\alpha(M(s))}{\sin^2\alpha(s)}-1;
$$
and let us expand by linearity
$$
\hat F_M(q,q,x,y)=\lim\limits_{\varepsilon\to 0} \hat F_M(q,q+\varepsilon, x,y).
$$
\begin{remark}\label{formulaBarF}
Note that $\sin^2\alpha(s)$ has a nice expression in terms of $q_1,q_2,x,y$.
$$
\sin^2\alpha(s)={\frac {{\left( q_1-q_2 \right) ^{2} \left( {x}^{2}+{y}^{2}+1
 \right) x}^{2} }{ \left( {(x-q_1)}^{2}+{y}^{2}{q_1}^{2}+y^{2}
 \right)  \left( {(x-q_2)}^{2}+{x}^{2}{q_2}^{2}+y^{2}
 \right)
}}
.
$$
Note also that if $q_1=q$ and $q_2=q+\varepsilon$, then we have
$$
\sin^2\alpha(s)={\frac {{x}^{2} \left( {x}^{2}+{y}^{2}+1 \right) }{\left({(x-q)}^{2} + {y}^{2}{
q}^{2}+{y}^{2}\right) ^{2}}}
\varepsilon^2 + O(\varepsilon^3).
$$
\end{remark}

Here we have similar definitions of asymptotic increase of $\sin^2\alpha$ (with respect to $\hat F$).

\begin{definition}
We say that a set of operators $\M$ {\it $q$-max-asymptotically increases} $\sin^2\alpha$
on a domain $D$ if
there exists $\varepsilon>0$ such that for
$(q,x,y)\in D$ there exists $M\in \M$
such that for any positive $\delta<\varepsilon$ we have
and any $(q,x,y)\in D$ we have
$$
\hat F_M(q,q+\delta,x,y)>0.
$$
\end{definition}

\begin{definition}
We say that a set of operators $\M$ {\it $q$-min-asymptotically increases} $\sin^2\alpha$
on a domain $D$ if
there exists $\varepsilon>0$ such that for any positive $\delta<\varepsilon$, any $M\in \M$
and any $(q,x,y)\in D$ we have
$$
\hat F_M(q,q+\delta,x,y)>0.
$$
\end{definition}

\begin{remark}
Note that
$$
\hat F(a,b,x,y)=F(a,b,y,x)
$$
Hence we have literally the same general statements for $\hat F(a,b,x,y)$ as we have shown for $F(a,b,x,y)$
(namely, Propositions~\ref{prop-a-1} and ~\ref{prop-a-2}). By that reason we omit their proofs here.
\end{remark}

\begin{proposition}\label{prop-q-1}
Let $\M$ be a family of operators continuously depending on $d$ parameters $($$d\ge 0$$)$;
each of the parameters is defined on a segment $[0,1]$.
Consider a domain  $D$ of triples $(q,x,y)$ in $\r^3$. Denote by $\bar D$ its closure in $\r P^3$.
Assume that for any $M\in \M$ we have the following:
\begin{itemize}
\item $\hat F_M(q,q,x,y)>0$ on $\bar D$ $($here and below we extend $\hat F_M$ by continuity to $\r P^3$$)$;
\item $\hat F_M(q,q,x,y)$ is defined on $\bar D$ $($i.e., has no zeroes of the polynomials in the denominator in appropriate coordinates$)$.
\end{itemize}
Then the family $\M$ $q$-min-asymptotically increases $\sin^2\alpha$
on $D$. \qed
\end{proposition}

\begin{proposition}\label{prop-q-2}
Let $\M$ be a finite set of operators.
Consider a domain  $D$ of triples $(q,x,y)$ in $\r^3$. Denote by $\bar D$ its closure in $\r P^3$.
Assume that for any  $M\in \M$ we have
\begin{itemize}
\item $\hat F_{M}(q,q,x,y)$ is defined on $\bar D$ $($here we extend $F_M$ by continuity to $\r P^3$$)$.
\item $\max\limits_{M\in\M} \hat F_{M}(q,q,x,y)$ is positive on $\bar D$.
\end{itemize}
Then the family $\M$ $q$-max-asymptotically increases $\sin^2\alpha$
on $D$.
\qed
\end{proposition}

Let us continue directly with formulation of the following important result.

\begin{proposition}\label{x-axes}
Let
$$
D''=\{(p,x,y){|\,}  p>-0.1,\,  x>y>1\},
$$
and $\M$ be the set of all admissible operators.
Then $\M$ $p$-max-asymptotically increases $\sin^2\alpha$ on $D''$.
\end{proposition}

\begin{proof}
We prove the statement separately for the following subsets $D_i$ of $D''$ and the following
sets $\M_i$.
The parametrisation of $D_i$ in all the cases is either $q=c$, $x$, $y$
with parameters $(c,x,y)$
or $q=c$, $x=c+kt$, $y=t$ with parameters $(c,k,t)$.

\begin{itemize}
\item Case 11 (see Lemma~\ref{inf-2-1}):
$\M_{11}$ $q$-max-asymptotically increases $\sin^2\alpha$ on $D_{11}$ where
$$
\begin{array}{l}
\M_{11}=\{W, V_{1,0;0}\};\\
D_{11}=\{(c,x,y){|\,}0\ge c\ge 1, y\ge 1,y+1>x>y\}.\\
\end{array}
$$

\item Case 12 (see Lemma~\ref{inf-2-2}):
$\M_{12}$ $q$-min-asymptotically increases $\sin^2\alpha$ on $D_{12}$ where
$$
\begin{array}{l}
\M_{12}=\{V_{0,0;k/3-\varepsilon}|1>\varepsilon\ge 0\};\\
 D_{12}=\{(c,c+kt,t){|\,}1\ge c\ge -0.1,k>3,t>1\}.\\
\end{array}
$$

\item Case 13 (see Lemma~\ref{inf-2-3}):
$\M_{13}$ $q$-min-asymptotically increases $\sin^2\alpha$ on $D_{13,1}$ and $D_{13,2}$, where
$$
\begin{array}{l}
\M_{13}=\{V_{0,0;1}\};\\
D_{13,1}=\{(c,c+kt,t){|\,}1\ge c\ge 0,3\ge k\ge 100/61,t>1\};\\
D_{13,2}=\{(c,c+kt,t){|\,}0\ge c\ge -0.1,3\ge k\ge 1, t>1\}.\\
\end{array}
$$

\item Case 14 (see Lemma~\ref{inf-2-4}):
$\M_{14}$ $q$-max-asymptotically increases $\sin^2\alpha$ on $D_{13}$ where
$$
\begin{array}{l}
\M_{14}=\{ V_{1,0;0},V_{0,0;1}\};\\
D_{14}=\{(c,c+kt,t){|\,}1\ge c\ge 0,5/3\ge k\ge  1,t>1\}.\\
\end{array}
$$

\item Case 15 (see Lemma~\ref{inf-3-1}):
$\M_{15}$ $q$-max-asymptotically increases $\sin^2\alpha$ on $D_{15}$ where
$$
\begin{array}{l}
\M_{15}=\{V_{1,0;0},V_{0,1;0}\};\\
D_{15}=\{(c,c+kt,t)){|\,}c\ge 1, k\le 1, t\ge 0\}.\\
\end{array}
$$

\item Case 16 (see Lemma~\ref{inf-3-2}):
$\M_{16}$ $q$-max-asymptotically increases $\sin^2\alpha$ on $D_{16}$ where
$$
\begin{array}{l}
\M_{16}=\{V_{a-\varepsilon-1,0;0}{|\,}1> \varepsilon\ge 0\};\\
 D_{16}=\{(c,c+kt,t){|\,} c\ge 1,k\ge 3,t>1,k+1>c>k-1\}.\\
\end{array}
$$

\item Case 17 (see Lemma~\ref{inf-3-3}):
$\M_{17}$ $q$-max-asymptotically increases $\sin^2\alpha$ on $D_{17}$ where
$$
\begin{array}{l}
\M_{17}=\{V_{1,0;1}\};\\
D_{17}=\{(c,c+kt,t){|\,} c\ge 1,\ge k\ge 1,t>1,k+1>c>k-1\}.\\
\end{array}
$$

\item Case 18 (see Lemma~\ref{inf-3-4}):
$\M_{18,i}$ $q$-max-asymptotically increases $\sin^2\alpha$ on $D_{18,i}$ for $i=1,2,3,4$ where
$$
\begin{array}{l}
\M_{18,1}=\{V_{0,0;k/3-\varepsilon}{|\,}1> \varepsilon\ge 0\};\\
\M_{18,2}=\{V_{0,0;1}{|\,}1> \varepsilon\ge 0\};\\
\M_{18,3}=\{V_{0,0;k/3-\varepsilon}{|\,}1> \varepsilon\ge 0\};\\
\M_{18,4}=\{V_{0,1;1}{|\,}1> \varepsilon\ge 0\};\\
D_{18,1}=\{(c,c+kt,t){|\,} c\le 0.81k, k\ge  3,t>1\};\\
D_{18,2}=\{(c,c+kt,t){|\,}c\le 0.81k, 3\ge k\ge  1,t>1;\\
D_{18,3}=\{(c,c+kt,t){|\,} 1.1k\ge c\ge 0.81k, k\ge  6,t>1\};\\
D_{18,4}=\{(c,c+kt,t){|\,} 1.1k\ge c\ge 0.81k, 6>k\ge  1,t>1\}.\\
\end{array}
$$

\item Case 19 (see Corollary~\ref{inf-3-6}):
$\M_{19}$ $q$-max-asymptotically increases $\sin^2\alpha$ on $D_{19}$ where
$$
\begin{array}{l}
\M_{19}=\{ V_{1,0;0}\};\\
D_{19}=\{(c,c+kt,t){|\,}c\ge 1,c\ge k+1,t>1\}.\\
\end{array}
$$

\end{itemize}

\begin{remark}
Similarly to several cases of Proposition~\ref{y-axes} (see Remark~\ref{flooor})
we use the floor function for Cases~12, 16, 17, 18 (subcases 1 and 3) to pick the integer operators.
\end{remark}

\vspace{2mm}

Let us make a brief outline of the cases with respect to a parameter $q=c$:
\begin{itemize}
\item $0\ge q\ge -0.1$ is covered by $D_{12}\cup D_{13,2}$;

\item $1\ge q\ge 0$ is covered by $D_{12} \cup D_{13,1} \cup D_{14}$ for $k\ge 1$
and by $D_{11}$ for $k\le 1$;

\item $q\ge 1$  is covered by

---  $D_{15 }$ if $k\le 1$;

--- $D_{18,1}\cup D_{18,2}\cup D_{18,3}\cup D_{18,4}$ if $1.1k\ge q\ge 1$;

--- $D_{16}\cup D_{17}$ if $k+1\ge q\ge k$;

--- $D_{19}$ if $q\ge k+1$.

\end{itemize}

Therefore,  $\M$ $q$-max-asymptotically increases $\sin^2\alpha$ on $D''$.
\end{proof}

\vspace{2mm}

{
\noindent{\it Proof of Theorem~\ref{greater e}~$($iv$)$.}
The statement of Theorem~\ref{greater e}~$($iv$)$ is equivalent to
the statement of Proposition~\ref{x-axes}.
}


\subsection{Study of Cases 1--19}\label{cases1-19}

In this subsection we study Cases~1-19 of Propositions~\ref{y-axes} and~\ref{x-axes}.

\begin{lemma}\label{inf-1-1}
Let
$$
D_1=\{(a,x,y)|a>1,x>y>1\}.
$$
Then the set
$$
\M_1=\{V_{1,0;0},V_{0,1;0}\}
$$
$p$-max-asymptotically increases $\sin^2\alpha$.
\end{lemma}

\begin{remark}
Note that both operators $V_{1,0;0}$ and $V_{0,1;0}$
are admissible for any $\xi$-state $s$.
\end{remark}

\begin{proof}
Note that the tangent of the slope of the line through $(0,a,1)$ and $\xi=(x,y,1)$
is smaller than $1$ (as $x>y$).
In case of the non-negative tangent  we can apply Corollary~\ref{gamma-3-corollary}
(and Corollary~\ref{gamma-4-corollary} for zero case) to conclude that $V_{0,1;0}$
increases the infinitesimal $\sin^2\alpha$.

In case of the negative tangent the line connecting points
$(0,a,1)$ and $\xi$ passes through  a point $(c,0,1)$ for some positive number $c$.
Here $V_{0,1;0}$  (respectively $V_{0,1;0}$) increases the infinitesimal $\sin^2\alpha$
if $c\le a$ (respectively if $a\le c$ ) by Corollary~\ref{gamma-1-corollary}.
From the conditions of Corollary~\ref{gamma-1-corollary}
it follows that for the case $x>y>1$ the positive real $\varepsilon$ can be taken to be equal to $1$.
\end{proof}

\begin{lemma}\label{inf-1-2}
Let
$$
D_2=\{(a,kt,a+t)|1\ge a\ge 0.52,k>1,t>1\}.
$$
Then
\begin{itemize}
\item The operator $V_{0,1;1}$ is admissible.
\item The set $\M_2=\{V_{0,1;1}\}$ $p$-min-asymptotically increases $\sin^2\alpha$.
\end{itemize}
\end{lemma}

\begin{proof}
{\it The first item.}
We start with a vector
$$
\xi=(kt,a+t,1).
$$
Then we have
$$
V_{0,1;1}\xi=(kt-(a+t)+1,a+t-1,1)
$$
It is clear that $kt-(a+t)$ is the difference of $x$ and $y$ coordinates of $\xi$ and, therefore, the first coefficient of the
new vector is positive. As $a+t>1$, the second coefficient of the new vector is positive as well. Hence this vector is
in the positive octant.

\vspace{2mm}

{\noindent
{\it The second item.}}
In order to make $D_2$ bounded we consider the following change of the coordinates:
$$
(a,k,t) \to (a,1/k,1/t).
$$
Now in the new coordinates we are interested in
$$
a\in [0.52,1]\,\qquad k\in[0,1], \qquad t\in[0,1].
$$
Set
$$
\begin{array}{l}
G_2(a,k,t)=F_{V_{0,1;1}}(a,a,kt,a+t);\\
H_2(a,k,t)=G_2(a,1/k,1/t).
\end{array}
$$
Then we have:
$$
H_2(a,k,t)=
\frac{h_2(a,k,t) }{
(a^2k^2t^2 + 2ak^2t + k^2t^2 + k^2 + 1)(a^2 + 2k^2 - 2a - 2k + 2)^2},
$$
where

$$
\begin{aligned}
h_2(a,k,t)=&
\big(a^6k^2 - 2a^2k^6 + 4a^4k^3 + 8a^2k^5 - 4ak^6 - 2a^4k^2 - 8a^3k^3 - 6a^2k^4 - k^6
\\&
 + 4a^3k^2 + 12a^2k^3 + 8k^5 - 4a^2k^2 - 8ak^3 - 6k^4 + 4ak^2 + 8k^3 - k^2\big)t^2
\\&
+\big(2a^5k^2 - 4ak^6 - 2a^5k + 4a^4k^2 + 4a^3k^3 + 8a^2k^4 + 14ak^5 - 4k^6 + 2a^4k
\\&
- 8a^3k^2 - 12a^2k^3 -  16ak^4 + 2k^5 - 4a^3k + 8a^2k^2 + 12ak^3 - 8k^4 + 4a^2k
\\&
- 4ak^2 + 4k^3 - 2ak - 4k^2 + 2k\big)t
\\&
+ a^4k^2 - 2k^6 - 2a^4k + 4a^3k^2 + 8ak^4 + 6k^5 - 6a^2k^2 - 8ak^3 - 11k^4
\\&
+ 4a^3 + 16ak^2 + 12k^3 - 6a^2 - 8ak - 12k^2 + 8a + 6k - 3.
\end{aligned}
$$
Direct computations  (see Case-2.mw in~\cite{maple}) show that the derivative of $h_2$ in variable $a$ is always positive,
so in order to find minimum we consider a new polynomial $h_{2,1}$ defined as
$$
h_{2,1}(k,t)=h_2(0.52,k,t).
$$
Let us estimate the derivatives of $h_{2,1}$ in variables $k$ and $t$.
We have
$$
\begin{array}{l}
\frac{\partial h_{2,1}}{\partial k}> -7;\\
\frac{\partial h_{2,1}}{\partial k}> -32\\
\end{array}
 $$
 (see details of computations in Case-2.mw of~\cite{maple}).
The final estimation for $h_{2,1}$ (see below) shows that $h_{2,1}>0.1$ on the grid,
hence $0.1$ is the admissible precession precision for showing positivity of $h_{2,1}$.
In other words, to show that
$$
h_{2,1}(k,t)>0
$$
in the region $[0,1]\times[0,1]$, it is sufficient to check that
the minimum of $h_{2,1}(k,t)$ at the nodes of the grid with steps
$$
\Delta k=\frac{0.1}{2\cdot 32}=\frac{1}{640};  \quad
\Delta t=\frac{0.1}{2\cdot 7}=\frac{1}{140}  \quad
$$
is greater than $0.1$.
(Note that the factor $1/2$ is as we have two variables $k$ and  $t$).

Direct computations (see Case-2.mw in~\cite{maple}) show that the
minimum over this grid is
$$
\frac{1563}{15625}=0.100032
$$
attained at the grid point
$$
a=\frac{13}{25}, \qquad  k=0, \qquad t= 0.
$$

Finally we check that the denominator of  $H_2$ does not have zeroes
on the closure of the region $[0.52,1]\times[0,1]\times[0,1]$  (and, therefore, it is bounded from below by some positive constant).
Hence $H_2(a,k,t)$ is bounded from below by some positive constant  in the region $[0.52,1]\times[0,1]\times[0,1]$.
This concludes the proof of Lemma~\ref{inf-1-2}.
(Note that all the estimations of this proof
checked directly by symbolic computations in MAPLE2020, see Case-2.mw in~\cite{maple}).
\end{proof}


\begin{lemma}\label{inf-1-3} $($i$)$
Let
\begin{itemize}
\item $D_{3,1}=\{(a,kt,a+t){|\,}0.33\ge a\ge 0,k\ge 3.1,t\ge 0.67\}$;

\item $D_{3,2}=\{(a,kt,a+t){|\,}0.52\ge a\ge 0,k\ge 3.1,t\ge 7/4\}$;

\item $D_{3,3}=\{(a,kt,a+t){|\,}0.1\ge a\ge -0.81, 1,k\ge 3.1,t\ge 0.9\}$.
\end{itemize}
Consider the family
$$
\M_3=\{V_{0,0;k/3-\varepsilon}|1>\varepsilon\ge 0\}.
$$
Then for the points in the domains $D_{3,1}$ and $D_{3,2}$ it holds:

$($i$)$ all the operators of $\M_3$ are admissible;

$($ii$)$ the family $\M_3$  $p$-min-asymptotically increases the value $\sin^2\alpha$.
\end{lemma}

\begin{proof}

We start with a vector
$$
\xi=(kt,a+t,1).
$$
Then we have
$$
V_{0,0; k/3-\varepsilon}\xi=\Big(kt-(k/3-\varepsilon)(a+t),a+t,1\Big).
$$
Since $a<1$, $k\ge 3.1$ and $t\ge 2/3$, the first coordinate is estimated as.
$$
kt-(k/3-\varepsilon)(a+t)>\frac{2}{3}kt-\frac{k}{3}a=\frac{k}{3}\big(2t-a\big)>0.
$$
Hence the resulting vector is in the non-negative coordinate octant.

\vspace{2mm}

Set
$$
\begin{array}{l}
H_3(a,k,t,\varepsilon)=F_ {V_{0,0;k/3-\varepsilon}}(a,a,kt,a+t);\\
\end{array}
$$
Then we have:
$$
\begin{array}{c}
H_3(a,k,t, \varepsilon)=
{\frac { \left( 3\,\varepsilon -k \right)  h_3(a,k,t) }{ \left( 9\,{
a}^{2}{k}^{2}+9\,{\varepsilon}^{2}+12\,\varepsilon\,k+4\,{k}^{2}+9 \right)
 \left( 9\,{a}^{2}{k}^{2}{t}^{2}+9\,{\varepsilon}^{2}{t}^{2}+12\,\varepsilon
\,k{t}^{2}+4\,{k}^{2}{t}^{2}+9\,{t}^{2} \right)  \left( {k}^{2}{t}^{2}
+{a}^{2}+2\,at+{t}^{2}+1 \right) }},
\end{array}
$$
where
$$
\begin{aligned}
h_3(a,k,t,\varepsilon) =&
\left( -9\,{a}^{6}{t}^{2}+36\,{a}^{5}{t}^{3}+45\,{a}^{4}{t}^{4}-18\,{
a}^{4}{t}^{2}+72\,{a}^{3}{t}^{3}-9\,{a}^{2}{t}^{2}+36\,a{t}^{3}-20\,{t
}^{4} \right) {k}^{5}+
\\&
\left( 27{a}^{6}\varepsilon{t}^{2}{+}54{a}^{5}
\varepsilon\,{t}^{3}{+}27{a}^{4}\varepsilon{t}^{4}{+}54{a}^{4}\varepsilon{t
}^{2}{+}108{a}^{3}\varepsilon{t}^{3}{+}27{a}^{2}\varepsilon{t}^{2}{+}54a
\varepsilon{t}^{3}{-}72\varepsilon{t}^{4} \right) {k}^{4}+
\\
&
 \left( -108\,
{a}^{4}{t}^{2}-108\,{a}^{3}{t}^{3}-81\,{\varepsilon}^{2}{t}^{4}-173\,{a}^
{2}{t}^{2}-58\,a{t}^{3}-65\,{t}^{4}-65\,{t}^{2} \right) {k}^{3}+
\\
&
 \left( -27\,{\varepsilon}^{3}{t}^{4}-99\,{a}^{2}\varepsilon\,{t}^{2}-90\,a
\varepsilon\,{t}^{3}-99\,\varepsilon\,{t}^{4}-99\,\varepsilon\,{t}^{2} \right)
{k}^{2}+
\\
&
\left( -81\,{a}^{2}{\varepsilon}^{2}{t}^{2}{-}162\,a{\varepsilon}^{2}
{t}^{3}{-}81\,{\varepsilon}^{2}{t}^{4}{-}99\,{a}^{2}{t}^{2}{-}144\,a{t}^{3}{-}81
\,{\varepsilon}^{2}{t}^{2}-45\,{t}^{4}-90\,{t}^{2} \right) k+
\\
&
(-27\,{a}^{2}{
\varepsilon}^{3}{t}^{2}{-}54\,a{\varepsilon}^{3}{t}^{3}{-}27\,{\varepsilon}^{3}{t}^
{4}{-}27\,{a}^{2}\varepsilon\,{t}^{2}{-}54\,a\varepsilon\,{t}^{3}{-}27\,{\varepsilon}
^{3}{t}^{2}{-}27\,\varepsilon\,{t}^{4}{-}54\,\varepsilon\,{t}^{2}).
\end{aligned}
$$

We are interested if  $H_3(a,k,t,\varepsilon)> 0$; this is equivalent to $h_3(a,k,t,\varepsilon)< 0$.
For the cases of $D_{3,1}$ and   $D_{3,2}$  the coefficients at $k^0$, $k^1$, and $k^2$  are non-positive, while
the coefficient at $k^3$ is negative.
The coefficients at $k^4$ and $k^5$ are non-positive for the ranges of the lemma.
For the case of $D_{3,3}$ all 6 coefficients are non-positive while some of them (e.g., at $k^5$) are always negative.
Finally we check that the denominator of  $H_3$ does not have zeroes
on the closure of all three regions (and, therefore, it is bounded from below by some positive constant).
All cases are checked by symbolic computations in MAPLE2020 (see Case-3.mw in~\cite{maple}).
\end{proof}


\begin{lemma}\label{inf-1-4}
Let
$$
D_{4}=\{(a,kt,a+t){|\,}0.68\ge a\ge 0.3,1001\ge k\ge 6.3,a+t>1\}.
$$
Consider the family
$$
\M_4=\{V_{(1-a)k-\delta ,1;k-\varepsilon-1}{|\,}1> \varepsilon\ge 0,1> \delta\ge 0\}.
$$
Then for the points in the domain $D_4$ it holds:

$($i$)$ all the operators of $\M_4$ are admissible;

$($ii$)$ the family $\M_4$  $p$-min-asymptotically increases the value $\sin^2\alpha$.
\end{lemma}

\begin{proof}
We start with a vector
$$
\xi=(kt,a+t,1).
$$
Then we have
$$
V_{(1-a)k-\delta ,1;k-\varepsilon-1}\xi=
(kt - (1 - a)k+\delta - (k-\varepsilon-1) (a + t-1) ,  a + t-1, 1).
$$
First of all, the second coordinate $a + t-1>0$ as $a+t>1$.
Secondly,
$$
kt - (1 - a)k+\delta - (k -\varepsilon - 1)(a + t-1) \ge
k(a+t-1) - (k -\varepsilon -1)(a + t-1)> 0.
$$
Hence the resulting vector is in the non-negative coordinate octant.

\vspace{2mm}

In order to make the domain of minimizing compact we consider the following change of the coordinates:
$$
(a,k,t) \to (a,1/k,1/t).
$$
Set
$$
\begin{array}{l}
G_4(a,k,t,\varepsilon,\delta)=F_{V_{(1-a)k-\delta ,1;k-\varepsilon}}(a,a,kt,a+t);\\
H_4(a,k,t,\varepsilon,\delta)=G_4(a,1/k,1/t,\varepsilon,\delta).
\end{array}
$$
(note that we take $V_{0,0;k-\varepsilon}$ with condition $2\ge \varepsilon \ge 1$ rather than
$V_{0,0;k-\varepsilon-1}$ with condition $1\ge \varepsilon \ge 0$. This is done in order to simplify symbolic computations.)
Then we have:
$$
H_4(a,k,t,\varepsilon,\delta)=
\frac { h_4(a,k,t,\varepsilon,\delta) }
{k^2(a^2k^2t^2 + 2ak^2t + k^2t^2 + k^2 + 1)(\delta^2 + \varepsilon^2 + 1)^2 },
$$
where
$$
h_4(a,k,t,\varepsilon,\delta) =c_4k^4+c_2k^2+c_0
$$
with coefficients
$$
c_0\ge\frac{11881}{9900}, \qquad c_2\ge -\frac{1544}{55}, \qquad c_4\ge -\frac{693231}{961}.
$$
Hence for $k\le 10/63$ the polynomial is positive.
Then  $H_4(a,k,t,\varepsilon,\delta)$ is positive as well.
Finally we check that the denominator of  $H_4$ does not have zeroes
on the closure of the region  (and, therefore, it is bounded from below by some positive constant).
Therefore, for $1001\ge k\ge 6.3$
the family $\M_4$  $p$-min-asymptotically increases the value $\sin^2\alpha$.

All the estimates are done symbolically  in MAPLE2020 (see Case-4.mw in~\cite{maple}).
\end{proof}

\begin{lemma}\label{inf-1-5}
Let
$$
 D_{5}=\{(a,kt,a+t){|\,}0.68\ge a\ge 0.3,k\ge 1000,7/4\ge t\ge 0.31\}.
 $$
Consider the family
$$
\M_5=\{V_{0 ,1;k/2-\varepsilon}{|\,}1> \varepsilon\ge 0\}.
$$
Then for the points in the domain $D_5$ it holds:

$($i$)$ all the operators of $\M_5$ are admissible;

$($ii$)$ the family $\M_5$  $p$-min-asymptotically increases the value $\sin^2\alpha$.
\end{lemma}

\begin{proof}
The proof of this lemma is similar to the proof of Lemma~\ref{inf-1-4}.
The resulting polynomial is
$$
h_{5}(a,k,t) =c_5k^5+c_4k^4+c_3k^3+c_2k^2+c_1k^1+c_0
$$
where the coefficients of $h_{5}$ do not exceed the respectively coefficients of the following polynomial
$$
-{\frac {40912\,{k}^{5}}{31}}-{\frac {2175788\,{k}^{4}}{4805}}-{\frac
{2619048\,{k}^{3}}{4805}}-{\frac {529992204\,{k}^{2}}{12641275}}-{
\frac {61022059\,k}{600625}}+{\frac{34569}{62500}}.
$$
Hence for $0\le k\le 1/1000$ the polynomial is positive.
The denominator of the corresponding expression does not have zeroes
on the closure of the region (and, therefore, it is bounded from below by some positive constant).
Therefore, for $k\ge 1000$ the set $\M_5$
$p$-min-asymptotically increases the value  $\sin^2\alpha$.

The estimates are done symbolically  in MAPLE2020 (see Case-5.mw in~\cite{maple})).
\end{proof}


\begin{lemma}\label{inf-1-6}
Let
$$
 D_{6}=\{(a,kt,a+t){|\,}0.68\ge a\ge 0.3,6.3\ge k\ge 3.13, a+t>1\}.
 $$
Consider the family
$$
\M_6=\{V_{1,1;3}\}.
$$
Then for the points in the domain $D_6$ it holds:

$($i$)$   the operator of $\M_6$ is admissible;

$($ii$)$ the family $\M_6$  $p$-min-asymptotically increases the value $\sin^2\alpha$.
\end{lemma}

\begin{proof}

We start with a vector
$$
\xi=(kt,a+t,1).
$$
Then we have
$$
V_{1,1;3}\xi=(kt - 3a - 3t + 2, a + t-1, 1),
$$
First of all, the second coordinate $a + t-1>0$ as $a+t>1$.
Secondly,
$$
kt - 3a - 3t + 2> 0,
$$
this expression attains minimum  $0.0003$ at
$$
a = 17/25, \qquad k = 313/100, \qquad  t = 31/100.
$$

Hence the resulting vector is in the positive coordinate octant.

\vspace{2mm}

In order to make the domain of minimizing compact we consider the following change of the coordinates:
$$
(a,k,t) \to (a,1/k,1/t).
$$
Further estimates are done symbolically  in MAPLE2020 (see Case-6.mw in~\cite{maple})).

\end{proof}


\begin{lemma}\label{inf-1-7} $($i$)$
Let
\begin{itemize}
\item $ D_{7,1}=\{(a,kt,a+t){|\,}0.68\ge a\ge 0,  2\ge k\ge 1,  t\ge 0.31\}$;

\item $D_{7,2}=\{(a,kt,a+t){|\,}0.68\ge a\ge 0,  3\ge k\ge 2,   t\ge 0.31\}$;

\item $D_{7,3}=\{(a,kt,a+t){|\,}0.7\ge a\ge 0.5,  3.3\ge k\ge 2.9  ,t\ge 0.29\}$;

\item $D_{7,4}=\{(a,kt,a+t){|\,}0.51\ge a\ge 0, 3.3\ge k\ge 2.9,  t\ge 0.4\}$;

\item $D_{7,5}=\{(a,kt,a+t){|\,}0\ge a\ge -0.81, 3.1\ge k\ge 1,  t\ge 1\}$.
\end{itemize}
Consider two one-operator sets:
$$
\M_{7,1}=\{V_{0,0;1}\} \quad \hbox{and} \quad \M_{7,2}=\{V_{0,1;1}\}.
$$
Then for the points in the domains $D_{7,1}$, $D_{7,2}$, $D_{7,4}$, and $D_{7,5}$
the operator of $\M_{7,1}$ is admissible
and the family $\M_{7,1}$  $p$-min-asymptotically increases the value $\sin^2\alpha$
in the domains.
\\
For the points in the domain $D_{7,3}$
the operator of $\M_{7,2}$ is admissible
and the family $\M_{7,2}$  $p$-min-asymptotically increases the value $\sin^2\alpha$
in the domains.
\end{lemma}

\begin{proof}
First of all note that both operators $V_{0,0;1}$ and $V_{0,1;1}$ send the whole set of admissible states to the positive octant.

Secondly, in all the cases  the corresponding operator increases the infinitesimal $\sin^2\alpha$; it was estimated symbolically in MAPLE2020
(the first two cases are in  Case-7-1-2.mw; the third case is in  Case-7-3.mw; the fourth case is in  Case-7-4.mw; the fifth case is in  Case-7-5.mw in~\cite{maple})).
\end{proof}


\begin{lemma}\label{inf-1-8}
Let
$$
 D_{8}=\{(a,x,y){|\,}1\ge a\ge 0.6799,2\ge y> 1, x>y\}.
 $$
Consider the family
$$
\M_8=\{V_{0,1;0}\}.
$$
Then for the points in the domain $D_8$ it holds:

$($i$)$ the operator of $\M_8$ is admissible;

$($ii$)$ the family $\M_8$  $p$-min-asymptotically increases the value $\sin^2\alpha$.
\end{lemma}

\begin{proof}
First of all the operator $V_{0,1;0}$ sends the whole set of admissible states to itself.

Now let us show that  $\{V_{0,1;0}\}$  $p$-min-asymptotically increases the value $\sin^2\alpha$.
In order to make the domain of minimizing compact we consider the following change of the coordinates:
$$
(a,xt,y) \to (a,1/x,1/y).
$$
Set
$$
\begin{array}{l}
G_8(a,k,t)=F_{V_{0,1;0}}(a,a,kt,a+t)+1;\\
H_8(a,k,t)=G_8(a,1/k,1/t).
\end{array}
$$
Symbolic estimations of the factors of
$$
\frac{\partial H_8(a,x,y)}
{\partial a}
$$
show that the derivative in $a$ variable is positive for the considered area.
Hence the minimums occur only when $a=0.6799$.

Set
$$
H_{8,1}(x,y)=H_8(0.6799,x,y).
$$
Symbolic estimations of the factors of
$$
\frac{\partial H_{8,1}(x,y)}
{\partial x}
$$
show that the derivative is negative for the considered area.
Hence the minimums occur only when $x=y$.

Finally after substitution $x=y$ we have a function of one variable.
Now the estimation of the minimum is straightforward.
Estimations show that is greater than $1.002$ (see Case-8.mw in~\cite{maple}).
\end{proof}


\begin{lemma}\label{inf-1-9}
Let
$$
 D_{9}=\{(a,kt,a+t){|\,}-0.8\ge a\ge -1.1,  k\ge 6,  a+t>1\}.
 $$
Consider the family
$$
\M_9=\{V_{(-a)k-\delta ,1;k/3 -\varepsilon -1}{|\,}1> \varepsilon\ge 0,1> \delta \ge 0\}.
$$
Then for the points in the domain $D_9$ it holds:

$($i$)$ all the operators of $\M_9$ are admissible;

$($ii$)$ the family $\M_9$  $p$-min-asymptotically increases the value $\sin^2\alpha$.
\end{lemma}

\begin{proof}
We start with a vector
$$
\xi=(kt,a+t,1).
$$
Then we have
$$
V_{(-a)k-\delta ,1;k/3 -\varepsilon -1}\xi=
(kt - (-ak)+\delta+ (k/3 -\varepsilon) (1-a - t) ,  a + t-1, 1).
$$
First of all, the second coordinate $a + t-1>0$ as $a+t>1$.
Secondly,
$$
kt - (-ak)+\delta+ (k/3 -\varepsilon) (1-a - t)\ge
(1.1k - (-ak)+\delta)
+
\big(
k(t-1.1)-(k/3-\varepsilon) (t-0.2)
\big)> 0.
$$
Hence the resulting vector is in the non-negative coordinate octant.

\vspace{2mm}

In order to make the domain of minimizing compact we consider the following change of the coordinates:
$$
(a,k,t) \to (a,1/k,1/t).
$$
Set
$$
\begin{array}{l}
G_9(a,k,t,\varepsilon,\delta)=F_{V_{(-a)k-\delta ,1;k/3 -\varepsilon -1}}(a,a,kt,a+t);\\
H_9(a,k,t,\varepsilon,\delta)=G_9(a,1/k,1/t,\varepsilon,\delta).
\end{array}
$$
Then we have:
$$
H_9(a,k,t,\varepsilon,\delta)=
\frac { h_9(a,k,t,\varepsilon,\delta) }
{(a^2 k^2 t^2 + 2 a k^2 t + k^2 t^2 + k^2 + 1) (9 \delta^2 k^2 + 9 \varepsilon^2 k^2 + 12 \varepsilon k + 9 k^2 + 4)^2},
$$
where
$$
h_9(a,k,t,\varepsilon,\delta) =c_5k^5+c_4k^4+c_3k^3+c_2k^2+c_1k+c_0
$$
where the coefficients of $h_{9}$ do not exceed the respectively coefficients of the following polynomial
$$
-540{k}^{5}-\frac{13041}{25}k^4-
\frac{697709}{1350}k^3-\frac{120917}{2025}k^2-
\frac{486791}{30000}k+\frac{953209}{90000}.
$$
The estimates are done symbolically  in MAPLE2020 (see caseI-3-negative-2.mw in~\cite{maple})).
Hence for $0\le k\le 1/6$ the polynomial is positive.
Finally we check that the denominator of  $H_9$ does not have zeroes
on the closure of all three regions (and, therefore, it is bounded from below by some positive constant).
Therefore,  the set $\M_9$
$p$-min-asymptotically increases the value  $\sin^2\alpha$ on $D_9$.

\end{proof}

\begin{lemma}\label{inf-1-10}
Let
$$
D_{10}=\{(a,kt,a+t)|-0.8\ge a\ge-1.1,6\ge k\ge 1,a+t>1\}.
$$
Then
\begin{itemize}
\item The operator $V_{1,0;1}$ is admissible.
\item The set $\M_{10}=\{V_{1,0;1}\}$ $p$-min-asymptotically increases $\sin^2\alpha$.
\end{itemize}
\end{lemma}

\begin{proof}
First of all note that the operator $V_{1,0;1}$  sends the whole set of admissible states to the positive octant.

Secondly, the increases of the infinitesimal $\sin^2\alpha$ was estimated symbolically in MAPLE2020
(see Case-10.mw in~\cite{maple})).
\end{proof}


\begin{lemma}\label{inf-2-1}
Let
$$
D_{11}=\{(c,x,y){|\,}1\ge c\ge 0, y\ge 1,y+1>x>y\}.
$$
Consider the set
$$
\M_{11}=\{W, V_{1,0;0}\}.
$$
Then for the points in the domain $D_{11}$  it holds:

$($i$)$ the operators of $\M_{11}$ are admissible;

$($ii$)$ the set $\M_{11}$  $q$-min-asymptotically increases the value $\sin^2\alpha$.
\end{lemma}

\begin{proof}
First of all the operator $V_{1,0;0}$ sends the whole set of admissible states to the positive octant.
For the operator $W$ we have
$$
W(x,y,1)=(x - y, y,  y + 1-x)
$$
Therefore, the admissible states of this lemma are sent to the positive octant.

In order to make the domain of minimizing compact we consider the following change of the coordinates:
$$
(c,x,y) \to (c,\rho, 1/y),
$$
where $\rho=x-y$. It is clear that $0<\rho<1$ and $0<1/y\le 1$.

Now we split the lemma in two cases.
First, for  $c\in [0, 0.8]$ we show that the operator $W$  increases the infinitesimal $\sin^2\alpha$.
Secondly, for $c\in [0.8, 1]$ we show that either the operator $W$
or the operator $V_{1,0;0}$  increases the infinitesimal $\sin^2\alpha$.
Both cases are checked by symbolic computations in MAPLE2020 (see Case-11-1.mw and Case-11-2.mw).
\end{proof}


\begin{lemma}\label{inf-2-2} $($i$)$
Let
$$
 D_{12}=\{(c,c+kt,t){|\,}1\ge c\ge -0.1,k>3,t>1\}.
$$
Consider the family
$$
\M_{12}=\{V_{0,0;k/3-\varepsilon}|1>\varepsilon\ge 0\}.
$$
Then for the points in the domain $D_{12}$ it holds:

$($i$)$ all the operators of $\M_{12}$ are admissible;

$($ii$)$ the family $\M_{12}$  $q$-min-asymptotically increases the value $\sin^2\alpha$.
\end{lemma}

\begin{proof}
The prof directly follows from the statement for the Case 3 with the domain $D_{3,3}$
(see Lemma~\ref{inf-1-3}).
\end{proof}


\begin{lemma}\label{inf-2-3}
Let
\begin{itemize}
\item $D_{13,1}=\{(c,c+kt,t){|\,}1\ge c\ge 0,3\ge k\ge 100/61,t>1\}$;

\item $D_{13,2}=\{(c,c+kt,t){|\,}0\ge c\ge -0.1,3\ge k\ge 1, t>1\}$.
\end{itemize}
Consider the set
$$
\M_{13}=\{V_{0,0;1}\}.
$$
Then for the points in the domains $D_{13,1}$  and $D_{13,2}$ it holds:

$($i$)$ the operator of $\M_{13}$ is admissible;

$($ii$)$ the family $\M_{13}$  $q$-min-asymptotically increases the value $\sin^2\alpha$.
\end{lemma}

\begin{proof}
First of all note that the operator $V_{0,0;1}$  send the whole set of admissible states to the positive octant.

\vspace{2mm}

Let us work out the domain $D_{13,1}$ first.
In order to make the domain of minimizing compact we consider the following change of the coordinates:
$$
(c,k,t) \to (c,1/k,1/t).
$$
Set
$$
\begin{array}{l}
G_{13}(c,k,t)=F_{V_{0,0;1}}(c,c,c+kt,t)+1;\\
H_{13}(c,k,t)=G_{13}(c,1/k,1/t).
\end{array}
$$
Here we are interested in the domain $t\in [0,1], k\in[0,0.61], c\in [0,1]$.

Symbolic estimations of the factors of $\frac{\partial H_{13}(c,k,t)}{\partial t}$
show that they do not change sign and that
$$
\frac{\partial H_{13}(c,k,t)}{\partial t}>0.
$$
Hence the minimum is attained at $t=0$.
Set
$$
H_{13,1}(c,k)=H_{13}(c,k,0)-1.
$$
Equivalently
$$
H_{13,1}(c,k)=
\frac{k(k - 2)(c^4k^4 - 2k^4 + 2k^3 - 3k^2 + 2k - 1)}{(k^2 + 1)(c^2k^2 + 2k^2 - 2k + 1)^2}.
$$

For $k>0$ we need to check positivity of $H_{13,1}$, which is equivalently to the fact that
$$
c^4k^4 - 2k^4 + 2k^3 - 3k^2 + 2k - 1<0.
$$
The last is checked symbolically. We also check that the denominator is bounded from below by some positive constant.
All the computations of this item are in Case-13-1.mw.

\vspace{2mm}

Now let us show that  $\M_{13}$  $q$-min-asymptotically increases the value $\sin^2\alpha$ on $D_{13,2}$.
In order to make the domain of minimizing compact we consider the following change of the coordinates:
$$
(c,k,t) \to (c,1/k,1/t).
$$
Note that:
$$
H_{13}(c,k,t)-1=
-\frac{kh_{13}(c,k,t)}{(c^2k^2t^2 + k^2t^2 + 2ckt + k^2 + 1)(c^2k^2 + 2k^2 - 2k + 1)^2},
$$
where
$$
h_{13}(c,k,t)=\lambda_1 t^2 +\lambda_2 t+\lambda_3.
$$
On the domain $D_{13,2}$ symbolic estimations show that
$$
\lambda_1\le -25/81; \quad \lambda_2<0.109; \quad \lambda_3<-0.6.
$$
(All the estimations for the case of  $D_{13,2}$  are in Case-13-2.mw.)
Therefore,
$$
h_{13,2}(c,k,t)<-0.49
$$
on $D_{13,2}$.
We also check that the denominator is bounded from below by some positive constant  (Case-13-2.mw).
Hence $H_{13}(c,k,t)-1$ is bounded from below on $D_{13,2}$ by some positive constant.
Therefore,  the family $\M_{13}$  $q$-min-asymptotically increases the value $\sin^2\alpha$ on $D_{13,2}$.
\end{proof}

\begin{lemma}\label{inf-2-4}
Let
$$
D_{14}=\{(c,c+kt,t){|\,}1\ge c\ge 0,5/3\ge k\ge  1,t>1\}.
$$
Consider the set
$$
\M_{14}=\{ V_{1,0;0},V_{0,0;1}\}.
$$
Then for the points in the domain $D_{14}$  it holds:

$($i$)$ the operators of $\M_{14}$ are admissible;

$($ii$)$ the set $\M_{14}$  $q$-min-asymptotically increases the value $\sin^2\alpha$.
\end{lemma}

\begin{proof}
First of all note that both operators $V_{0,0;1}$ and $V_{1,0;0}$ send the whole set of admissible states to the positive octant.

\vspace{2mm}

Further we check the factors of the difference of infinitesimal $\sin^2\alpha$ for the original and the resulting
infinitesimal states. Here we use symbolic computations of MAPLE2020  (Case-14.mw) to estimate the derivatives of the functions and then find minimum with the small enough grid to assure the positivity of the difference.
We estimate the maximum of two polynomial expressions.
Here we estimate maxima/minima (depending which is closer to zero) of derivatives of these two expressions in order to determine the step of the grid.
For both operators we check that the denominators of the expression do not have roots at the closure of the domains.
(See Case-14.mw.)
\end{proof}


\begin{lemma}\label{inf-3-1}
Let
$$
D_{15}=\{(c,c+kt,t)){|\,}c\ge 1, k\le 1, t\ge 0\}.
$$
Then the set
$$
\M_{15}=\{V_{1,0;0},V_{0,1;0}\}
$$
$q$-max-asymptotically increases $\sin^2\alpha$.
\end{lemma}

\begin{proof}
Note that the tangent of the slope of the line through $(c,0,1)$ and $\xi=(x,y,1)$
is smaller than or equal to $1$ (as $k\le 1$).
In case if the non-negative tangent we can apply Corollary~\ref{gamma-3-corollary}
(and Corollary~\ref{gamma-4-corollary} for zero case) to conclude that $V_{1,0;0}$
increases the infinitesimal $\sin^2\alpha$.

In case of the negative tangent  the line connecting points
$(c,0,1)$ and $\xi$ passes through  a point $(0,a,1)$ for some positive number $a$.
Here $V_{1,0;0}$  (respectively $V_{0,1;0}$) increases the infinitesimal $\sin^2\alpha$
if $c\le a$ (respectively if $a\le c$ ) by Corollary~\ref{gamma-1-corollary}.
From the conditions of Corollary~\ref{gamma-1-corollary}
it follows that for the case $x>y>1$ the positive real $\varepsilon$ can be taken to be equal to $1$.
\end{proof}

\begin{lemma}\label{inf-3-2}
Let
$$
 D_{16}=\{(c,c+kt,t){|\,} c\ge 1,k\ge 3,t>1,k+1>c>k-1\}.
 $$
Consider the family
$$
\M_{16}=\{V_{c-\varepsilon-1,0;0}{|\,}1> \varepsilon\ge 0\}.
$$
Then for the points in the domain $D_{16}$ it holds:

$($i$)$ all the operators of $\M_{16}$ are admissible;

$($ii$)$ the family $\M_{16}$  $q$-min-asymptotically increases the value $\sin^2\alpha$.
\end{lemma}

\begin{proof}
First of all note that
for a vector
$$
\xi=(c+kt,t,1)
$$
we have
$$
V_{c-\varepsilon-1,0;0}\xi=(kt + c - (c-\varepsilon)+1, t, 1);
$$
the resulting vector is in the positive octant.

Set
$$
\begin{array}{l}
H_{16}(c,k,t,\varepsilon)=F_{V_{c-\varepsilon-1,0;0}}(c,c,c+kt,t)+1;
\end{array}
$$

For the second item we check that the derivative of $H_{16}$ in $t$ is positive, and after the substitution $t=1$
the derivative in $c$ is positive, hence we substitute $c=k-1$.
Finally the expression is simple enough to apply symbolic computations of MAPLE2020
and to show it always exceeds 1 (see Case-16.mw for more details).
\end{proof}

\begin{lemma}\label{inf-3-3}
Let
$$
 D_{17}=\{(c,c+kt,t){|\,} c\ge 1,3\ge k\ge 1,t>1,k+1>c>k-1\}.
 $$
Consider the family
$$
\M_{17}=\{V_{1,0;1}\}.
$$
Then for the points in the domain $D_{17}$ it holds:

$($i$)$ the operator of $\M_{17}$ is admissible;

$($ii$)$ the set $\M_{17}$  $q$-min-asymptotically increases the value $\sin^2\alpha$.
\end{lemma}

\begin{proof}
First of all note that
for a vector
$$
\xi=(c+kt,t,1)
$$
we have
$$
V_{1,0;1}\xi=(kt + c - t - 1, t, 1).
$$
The resulting vector is in the positive octant.

\vspace{2mm}

Let us discuss the second item.
In order to make the domain of minimizing compact we consider the following change of the coordinates:
$$
(c,k,t) \to (c,k, 1/t).
$$
Set
$$
\begin{array}{l}
G_{17}(c,k,t)=F_{V_{1,0;1}}(c,c,c+kt,t);\\
H_{17}(c,k,t)=G_{17}(c,k,1/t).
\end{array}
$$
Now we check the factors of $H_{17}$ .
Here we use symbolic computations of MAPLE2020 to estimate the derivatives of the functions
and then and then find minimum with the small enough grid to assure the positivity of  $H_{17}$ (see Case-17.mw).
\end{proof}

\begin{lemma}\label{inf-3-4}
Let

\begin{itemize}
\item $D_{18,1}=\{(c,c+kt,t){|\,} c\le 0.81k, k\ge  3,t>1\}$;

\item $D_{18,2}=\{(c,c+kt,t){|\,}c\le 0.81k, 3\ge k\ge  1,t>1$;

\item $D_{18,3}=\{(c,c+kt,t){|\,} 1.1k\ge c\ge 0.81k, k\ge  6,t>1\}$;

\item $D_{18,4}=\{(c,c+kt,t){|\,} 1.1k\ge c\ge 0.81k, 6>k\ge  1,t>1\}$.

 \end{itemize}

Consider the families

\begin{itemize}
\item $\M_{18,1}=\{V_{0,0;k/3-\varepsilon}{|\,}1> \varepsilon\ge 0\}$;

\item $\M_{18,2}=\{V_{0,0;1}{|\,}1> \varepsilon\ge 0\}$;

\item $\M_{18,3}=\{V_{0,0;k/3-\varepsilon}{|\,}1> \varepsilon\ge 0\}$;

\item $\M_{18,4}=\{V_{0,1;1}{|\,}1> \varepsilon\ge 0\}$.
\end{itemize}

Then for $i=1,2,3,4$ we have that the points in the domain $D_{18_i}$ satisfy

$($i$)$ the operators of $\M_{18,i}$ covers $D_{18,i}$ is admissible;

$($ii$)$ the sets $\M_{18,i}$  $q$-min-asymptotically increases the value $\sin^2\alpha$ on $D_{18,i}$.
\end{lemma}

\begin{proof}
The proof is complete after the change of the coordinates ($a$ to $c$) in
Lemma~\ref{inf-1-3} (Case 3, item 3),
Lemma~\ref{inf-1-7} (Case 7, item 5),
Lemma~\ref{inf-1-9} (Case 9), and
Lemma~\ref{inf-1-10} (Case 10).
\end{proof}

\begin{corollary}\label{inf-3-6}
Let
$$
D_{19}=\{(c,c+kt,t){|\,}c\ge 1,c\ge k+1,t>1\}.
$$
Consider the set
$$
\M_{19}=\{ V_{1,0;0}\}.
$$
Then for the points in the domain $D_{19}$  it holds:

$($i$)$ the operator of $\M_{19}$ is admissible;

$($ii$)$ the set $\M_{19}$  $q$-min-asymptotically increases the value $\sin^2\alpha$.
\qed
\end{corollary}

\begin{proof}
The statement follows directly from Proposition~\ref{fin-3-5}.
\end{proof}

\begin{remark}
Note that it is not sufficient to consider only $V_{*,*;*}$ transformations
in order to increase the value of $\sin^2\alpha$ (in the asymptotical case $q_1=q_2$).
For instant one can consider
$$
(c,x,y)=\Big(\frac{2}{3},\frac{3}{2},\frac{11}{8}\Big).
$$
Here it is sufficient  to test all admissible composition of operators $V_{*,*;*}$.
They are  as follows:
$$
\begin{array}{l}
V_{1,1;0}, \quad
V_{1,0;0}, \quad
V_{0,1;0}, \quad
V_{0,0;1}, \quad
V_{0,1;0}V_{0,0;1}, \quad
V_{1,1;1}, \\
V_{1,0;0}V_{0,1;1}, \quad
V_{0,1;1}V_{0,0;1}, \quad
\hbox{and}
\\
V_{0,1;n}, \hbox{ for $n=1,2,3,4$}
\end{array}
$$
 (see the computations in non-V.mw in~\cite{maple}).
 All the listed operators decrease  the value of $\sin^2\alpha$.
For that reason we have one case (Case~11) where the operator $W$
is considered.
\end{remark}


\section{Some open questions}\label{conjectures}

There are several improvements that can be done with the $\sin^2$-algorithm.

First of all, it is most likely that the algorithm will be periodic for algebraic vectors if we remove
the JP-transformation $W$.

\begin{conjecture}
Exclusion of JP-transformation $W$ from admissible transformations
will not affect periodicity for cubic vectors.
\end{conjecture}

Secondly, it is most likely that Preliminary Stage 2 is excessive, as
the iterations of the main stage applied to a supporting basis
bring us to a separating basis in finitely many steps in all computed examples.

\begin{conjecture}\label{conj-2}
Exclusion of Preliminary Stage 2 in the $\sin^2$-algorithm
will not affect periodicity for cubic vectors.
\end{conjecture}

Finally it remains to recall the remaining open cases of Hermite's problem.

\begin{problem}
Find a generalized Euclidean algorithm that is periodic for real cubic vectors in the non-totally-real case.
\end{problem}

Here the analytic extension of the formula for $\sin^2\alpha$ to the complex numbers may be of use.

\vspace{2mm}

Finally there is almost nothing known for the higher dimensional cases.
\begin{problem}
Find a generalized Euclidean algorithm that is are periodic for $n$-algebraic vectors in $\r^n$
for $n\ge 4$.
\end{problem}

The last two problems (including a heuristic algorithm that works in both cases)
are discussed in~\cite{Karpenkov2021-2}.

{\noindent
{\bf Acknowledgements.}
The author is grateful to  to P.~Giblin, A.~Pratoussevitch, H.~\v{R}ada for useful comments and to  A.~Ustinov for useful discussions.
The author is thankful to his wife Tanya for care and support while completing this article during the lockdown.
}

\bibliographystyle{plain}
\bibliography{cfbiblio}

\vspace{.5cm}


\end{document}